\documentclass[12pt]{amsart}
\usepackage{comment}
\usepackage[scale=0.75]{geometry} 
\usepackage{amsaddr}
\usepackage{amsmath}
\usepackage{amssymb}
\usepackage{latexsym}
\usepackage{amsthm}
\usepackage{mathrsfs}
\usepackage{color}
\usepackage{amsfonts}
\usepackage{enumitem}
\usepackage{array}
\usepackage{setspace}
\usepackage{mathtools}
\usepackage{tikz}
\usepackage{hyperref}
\usetikzlibrary{arrows}
\usetikzlibrary{knots,patterns}

\usepackage{stmaryrd,graphicx}

\newtheorem{thm}{Theorem}[section]
\newtheorem*{thm*}{Theorem}
\newtheorem{cor}[thm]{Corollary}
\newtheorem*{cor*}{Corollary}
\newtheorem{lem}[thm]{Lemma}
\newtheorem*{lem*}{Lemma}
\newtheorem{prop}[thm]{Proposition}
\newtheorem*{prop*}{Proposition}

\theoremstyle{definition}
\newtheorem{defn}[thm]{Definition}
\newtheorem*{defn*}{Definition}
\newtheorem{conjecture}{Conjecture}
\newtheorem*{conjecture*}{Conjecture}
\newtheorem*{condition*}{Condition}
\newtheorem*{assumption*}{Assumption}

\theoremstyle{remark}
\newtheorem{rem}[thm]{Remark}
\newtheorem*{rem*}{Remark}

\newtheorem*{notation}{Notation}

\newtheorem*{problem*}{Problem}

\numberwithin{equation}{section}

\newcommand{\BQ}{\mathbb Q}

\newcommand{\BC}{\mathbb C}

\newcommand{\BZ}{\mathbb Z}
\newcommand{\BP}{\mathbb P}

\newcommand{\CC}{\mathcal C}
\newcommand{\CD}{\mathcal D}
\newcommand{\CE}{\mathcal E}
\newcommand{\CF}{\mathcal F}

\newcommand{\CO}{\mathcal O}

\newcommand{\CT}{\mathcal T}

\newcommand{\wt}{\widetilde}

\newcommand{\eps}{\epsilon}
\newcommand{\veps}{\varepsilon}

\newcommand{\chern}{\mathrm{Chern}}
\newcommand{\verlinde}{\mathrm{Ver}}

\newcommand{\Z}{\mathbb Z}
\newcommand{\rk}{\mathrm{rk}}

\DeclareMathOperator{\ch}{ch}

\DeclareMathOperator{\Li}{Li}

\DeclareMathOperator{\res}{res}

\DeclareMathOperator{\pExp}{Exp}

\DeclareMathOperator{\Hilb}{Hilb}
\DeclareMathOperator{\Pic}{Pic}
\newtheorem*{acknowledgements}{Acknowledgements}

\newenvironment{LG}{\noindent \color{red}{\bf LG:} \footnotesize}{}

\title[Verlinde and Segre formula]{Refined Verlinde and Segre formula for Hilbert schemes}
\author{Lothar G\"ottsche}
\email{gottsche@ictp.it}
\address{International Centre for Theoretical Physics,\\
 Strada Costiera 11, 34151 Trieste, Italy}
\author{Anton Mellit}
\email{anton.mellit@univie.ac.at}
\address{Faculty of Mathematics, University of Vienna, \\
Oskar-Morgenstern-Platz 1, 1090 Vienna, Austria}

\begin{document}
\onehalfspacing

\begin{abstract}
Let   $\Hilb_nS$ be the Hilbert scheme of $n$ points on a smooth projective surface $S$.
To a class $\alpha\in K^0(S)$  correspond a tautological vector bundle $\alpha^{[n]}$ on $\Hilb_nS$ and  line bundle
$L_{(n)}\otimes E^{\otimes r}$ with $L=\det(\alpha)$, $r=\rk(\alpha)$. In this paper we give closed formulas for the  generating functions for the
Segre classes $\int_{\Hilb_nS} s(\alpha^{[n]})$, and the Verlinde numbers $\chi(\Hilb_nS,L_{(n)}\otimes E^{\otimes r})$, for any surface $S$ and  any class $\alpha\in K^0(S)$.
In fact we determine a more general generating function for $K$-theoretic invariants of Hilbert schemes of points, which contains the formulas for Segre and Verlinde numbers as
specializations.
We prove these formulas in case $K_S^2=0$. Without assuming the condition $K_S^2=0$, we show the Segre-Verlinde conjecture of Johnson and Marian-Oprea-Pandharipande, which relates the Segre and Verlinde generating series by an explicit change of variables.
\end{abstract}

\date\today
\maketitle

\tableofcontents

\section{Introduction}
Let $S$ be a smooth projective surface. Among the most basic moduli spaces associated to $S$, with connections to many other moduli spaces,  are the Hilbert schemes $\Hilb_nS$ of $n$ points on $S$,
which have been a focus of interest for many years. In particular the following two important series of enumerative invariants of the  Hilbert schemes of points have been studied for more than 2 decades.

(1) Chern and Segre series of tautological bundles.\\
Let $Z_n(S)\subset S\times \Hilb_nS$ be the universal subscheme, with projections $p:Z_n(S)\to \Hilb_nS$, $q:Z_n(S)\to S$.
For  a vector bundle $V$ on $S$, the corresponding tautological bundle is $V^{[n]}:=p_*q^*(V)$, a vector bundle of rank $n\rk(V)$ on $\Hilb_nS$.
This extends to a homomorphism $\bullet^{[n]}:K^0(S)\to K^0(\Hilb_nS)$ of Grothendieck groups of vector bundles.
We consider  the series of Chern integrals
\[
I_{S,\alpha}^C(x):=\sum_{n\ge 0} x^n \int_{\Hilb_nS} c_{2n}(\alpha^{[n]}),
\]
for any class $\alpha\in K^0(S)$, and  the corresponding Segre series
\[
I_{S,\alpha}^S(x):=\sum_{n\ge 0} x^n \int_{\Hilb_nS} s_{2n}(\alpha^{[n]}).
\]
We will always write $k=\rk(\alpha)$. Because of the obvious identity  $I_{S,\alpha}^S(x)=I_{S,-\alpha}^C(x)$ we will concentrate on
$I_{S,\alpha}^C(x)$.

(2)  The Verlinde series.\\
Let $\sigma:S^n\to S^{(n)}$ be the quotient map to the symmetric power, and $\pi:\Hilb_nS\to S^{(n)}$ the Hilbert-Chow morphism.
For a line bundle $L\in \Pic(S)$ let $L_{(n)}:=\pi^*\sigma_*(\otimes_{i=1}^n pr^*_iL)^{\mathfrak S_n}$ be the pullback of the symmetrized pushforward.
Furthermore for $n\ge 2$ let $E:=-\frac{1}{2}D$, where $D$ is the exceptional divisor of $\pi$.
Then it is well-known that $\Pic(\Hilb_nS)=\Pic(S)_{(n)}\oplus \Z E$.
Furthermore for $\beta\in K^0(S)$ we have
$\det(\beta^{[n]})=L_{(n)}\otimes E^{\otimes r}$  with $\det(\beta)=L$ and $\rk(\beta)=r$.
The Verlinde series associated to $(S,\beta)$ is the generating series of the holomorphic Euler characteristics
\[
I^V_{S,\beta}(t)=\sum_{n\ge 0} t^n \chi(\Hilb_nS,\det(\beta^{[n]})),
\]
where we always write $r=\rk(\beta)$.
Finally when $k$ and $r$ occur together we will always have $r=k-1$.

For $\alpha=L$ a line bundle, the coefficients of the Segre series $I_{S,L}^S(z)$ were already considered in
\cite{ellingsrudbott} in the case $S=\BP^2$; the first 8 terms were computed and related to the degrees of the varieties of sums of powers of ternary linear forms, to the counting of
Darboux curves and to Donaldson invariants of $\BP^2$.
The well-known Lehn conjecture \cite{lehnchern} is  a conjectural formula for  $I_{S,L}^S(z)$ for any surface and any line bundle $L\in \Pic(S)$.
It was first proven in \cite{marianSegre} in the special case of $K$-trivial surfaces, and then in general in \cite{voisinSegre},
\cite{marian2019combinatorics}. Finally, in \cite{marian2017higher}, the case of $I_{S,\alpha}^C(x)=I_{S,-\alpha}^S(x)$ for an arbitrary class $\alpha\in K^0(S)$ is considered:
Applying the cobordism invariance of \cite{ellingsrud2001cobordism}, one can write  $I_{S,\alpha}^C(x)$   in the following form
\begin{align*}
I_{S,\alpha}^C(x)&= A_0(x)^{c_2(\alpha)} A_1(x)^{\chi(\det(\alpha))} A_2(x)^{\frac{1}{2}\chi(\CO_S)} A_3(x)^{c_1(\alpha)K_S -\frac{1}{2}K_S^{2}}  A_4(x)^{K_S^2}.
\end{align*}
Furthermore, by  explicitly determining the invariants for $K3$ surfaces, they determined
$ A_0$, $ A_1$, $ A_2$  as  algebraic functions.
Finally they proved explicit formulas as algebraic functions for $A_3(x)$ and $ A_4(x)$ for  $k=-1$, $k=-2$.
In case $k=0$ they showed (in our notation) that $A_4=1$ and conjectured  $A_3=1$, which was recently proved  in \cite{Yuan2022}.

Identifying a subscheme  $Z\in \Hilb_nS$ with its ideal sheaf $I_{Z}$, the Hilbert scheme $\Hilb_nS$ can be viewed as a moduli space of stable  rank $1$  torsion free sheaves on $S$ with Chern classes $c_1=0$ and $c_2=n$. Thus a formula for the Verlinde series $I^V_{S,\beta}(t)$ can be viewed as the rank $1$ case of a  general Verlinde formula for surfaces, computing the holomorphic Euler characteristics of all determinant bundles on all moduli spaces of sheaves on $S$. This is the higher dimensional  analogue of the famous Verlinde formula \cite{Verlinde} for curves, which computes the dimensions of the spaces of sections of
determinant bundles on moduli spaces of vector bundles on curves.

By \cite{ellingsrud2001cobordism}  the Verlinde series $I^V_{S,\beta}(t)$ has  for all surfaces $S$ and all $\beta\in K^0(S)$ the factorization (recall $L=\det(\beta)$)
\[
I^V_{S,\beta}(t)=B_1(t)^{\chi(L)} B_2(t)^{\frac{1}{2}\chi(\CO_S)}B_3(t)^{LK_S-\frac{1}{2}K_S^2}B_4(t)^{K_S^2}.
\]
With the change of variables $t=-y(1-y)^{r^2-1}$ we have
\[
B_1(t)=1-y, \quad B_2(t)= \frac{(1-y)^{r^2}}{(1-r^2y)};
\]
Furthermore $B_3=B_4=1$ for $r=-1,0,1$, and replacing $r$ by $-r$ sends $B_3(t)$ to $\frac{1}{B_3(t)}$ and $B_4(t)$ to itself.

Based on Le Potier's strange duality conjecture for surfaces \cite{LePotierStrange},  in \cite{JohnsonHilbert} the Segre and Verlinde series are related to each other by a change of variables, whose explicit form was determined in \cite{marian2019combinatorics}. This  gives the following conjectural Verlinde-Segre correspondence.
\begin{conjecture}
\label{VerSeg}\cite{JohnsonHilbert}, \cite{marian2019combinatorics}.
Fix $k=r+1$ then
\[
 A_3(x)=B_3(t), \quad  A_4(x)=B_4(t),
\]
under the change of variables
\[
x=s(1-rs)^{-r}, \quad t=\frac{s(1-(r-1)s)^{r^2-1}}{(1-rs)^{r^2}}.
\]
\end{conjecture}

In this paper we prove a closed formula for $ A_3(x)$ and $B_3(t)$  and give a conjectural formula for
 $ A_4(x)$ and $B_4(t)$ for arbitrary $k=r+1$. Thus we obtain complete conjectual Segre and Verlinde formulas for Hilbert schemes of points on any surface $S$,
 and we prove them when $K_S^2=0$. Furthermore we prove the Segre-Verlinde correspondence in general.
 In addition our methods give independent proofs for the formulas $ A_0(x)$, $ A_1(x)$, $  A_2(x)$, of \cite{marian2017higher}, and also express them in a slightly simpler form.

We obtain these formulas as specializations of a considerably more general result. We introduce a  $K$-theoretic invariant of Hilbert schemes of points, which has
 both the Verlinde and Segre invariants as specializations, and determine its generating function.  The general shape of this generating function then implies in particular the
 Verlinde-Segre correspondence, and to some extend explains it.
For $\alpha\in K^0(S)$ we put
 \begin{equation}\label{chilambda}
I_{S,\alpha}(w,z):=\sum_{n\ge 0} (-w)^n\chi\left(\Hilb_nS, \left(\Lambda_{-z} \alpha^{[n]}\right)\otimes \det(\CO_S^{[n]})^{-1}\right).
\end{equation}
Here $\Lambda_{-z}: (K^0(S),+)\to (K^0(S)[[z]],\cdot)$ is the homomorphism given by
\[
\Lambda_{-z} (W)=\sum_n (-z)^n \Lambda^{n}W, \qquad \Lambda_{-z} (-W)=\sum_n z^n S^nW,
\]
for $W$ a vector bundle.
It is easy to see that both
$I_{S,\alpha}^C(z)$ and $I_{S,\alpha-\CO_S}^V(t)$ are specializations of $I_{S,\alpha}(w,z)$, in fact for $k=\rk(\alpha)$ we have
\begin{equation}\label{Ispecial}
\begin{split} I_{S,\alpha}^C(x)=I_{S,\alpha}\Big(-\eps^{2-k}(1+\eps)^kx,\frac{1}{1+\eps}\Big)\Big|_{\eps^0},\qquad
I_{S,\alpha-\CO_S}^V(t)=I_{S,\alpha}(-tz^{-k},-z)|_{z^0}.
\end{split}
\end{equation}
Similar expressions to \eqref{chilambda} have been considered in  \cite{Arbesfeld2019} and \cite{Arbesfeld2022}, and very similar expressions were
studied in \cite{Bojko2021} for virtual invariants of Quot schemes of curves, surfaces and $4$-folds and  used to understand the Verlinde-Segre correspondence for these virtual invariants.
The following is our main result.
\begin{thm}
Fix $k\in \Z$. Then there are power series
$G_0,G_1,G_2,G_3,G_4\in \Z[[w,z]]$, such that for all smooth projective surfaces $S$ and all $\alpha\in K^0(S)$ of rank $k$ we have
 \begin{equation}\label{prodform}
I_{S,\alpha}(w,z)=G_0(w,z)^{c_2(\alpha)}G_1(w,z)^{\chi(\det(\alpha))}G_2(w,z)^{\frac{1}{2}\chi(\CO_S)}G_3(w,z)^{c_1(\alpha)K_S -\frac{1}{2}K_S^{2}} G_4(w,z)^{K_S^2}.
\end{equation}
With the change of variables
 \begin{equation}\label{varchange}  w=\frac{u(1-u)^{k-1}}{v(1-v)^{k-1}}, \qquad z=\frac{v}{(1-u)^{k-1}}, \qquad\end{equation}
 we have
 \begin{align*}
 G_0(w,z)&=\frac{(1-u-v)^{k}}{(1-v)^{k-1}\big((1-u)^{k-1}-v\big)},\\
  G_1(w,z)&=\frac{(1-v)^{k-2}\big((1-u)^{k-1}-v\big)}{(1-u)(1-u-v)^{k-1}},\\
 G_2(w,z)& =\frac{(1-\frac{u}{v})^2(1-v)^{(k-2)^2}\big((1-u)^{k-1}-v\big)^{2(k-1)} }{(1-u-v)^{(k-1)^2}(1-u)^{k^2-2k}\big(1-u-v-(k^2-2k)uv\big)}.
\end{align*}
With the further change of variables $y=\frac{uv}{(1-u)(1-v)}$, we have that
$$G_3(w,z)=(1-y)^{-\frac{k-1}{2}}\exp\left(\sum_{n=1}^\infty-\frac{y^n}{2n}\left.\left(\frac{x^{k-1}-x^{1-k}}{x-x^{-1}}\right)^{2n}\right|_{x^0} \right),$$
and $G_4(w,z)\in \BQ[[y]]$.
\end{thm}
Here the product formula \eqref{prodform} is a direct application of \cite[Thm.~4.2]{ellingsrud2001cobordism}, and the substance of the theorem are the formulas for the universal power series $G_i(w,z)$.
 Specializing, we obtain the Verlinde-Segre correspondence and the Verlinde and Segre formulas for Hilbert schemes.
The Verlinde-Segre correspondence follows, because $B_3(t)$, $ A_3(x)$ and $B_4(t)$, $ A_4(x)$  are explicit specializations of  $G_3(w,z)$, and  $G_4(w,z)$, together with the fact that after the changes of variables \eqref{varchange} and $y=\frac{uv}{(1-u)(1-v)}$, (or equivalently $u=\frac{-(1-v^{-1})y}{1-(1-v^{-1})y}$),
both  $G_3(w,z)$ and $G_4(w,z)$ depend only on $y$.
 \begin{cor}\label{VerSegCor} Conjecture \ref{VerSeg} is true. Explicitly we have the following. Assume $r=k-1$, then under the changes of variables \eqref{varchange}, $y=\frac{uv}{(1-u)(1-v)}$, and
 \begin{equation} \label{varchange1} x=-y(1-ry)^{r-1},\quad t=-y(1-y)^{r^2-1}
 \end{equation}
 we have
 \begin{align*}
 B_1(t)&=  A_0(x)  A_1(x)=G_0(w,z)G_1(w,z),\\
 B_3(t)&=  A_3(x)=G_3(w,z),\quad B_4(t)= A_4(x)=G_4(w,z).
 \end{align*}
 \end{cor}
Note that Conjecture \ref{VerSeg} follows from Corollary \ref{VerSegCor}  by putting
 $s=\frac{-y}{1-ry}$.
\begin{cor}
Fix $r=k-1\in \Z$.
 Under the changes of variables
\ref{varchange1} we have
\begin{align*}
  A_0(x)&=\frac{(1-y)^{r+1}}{1-ry},\quad
  A_1(x)=\frac{1-ry}{(1-y)^r},\quad
  A_2(x)=\frac{(1-ry)^{2r}}{(1-y)^{r^2}(1-r^2y)},\\
B _1(t)&=1-y,\quad B_2(t)=\frac{(1-y)^{r^2}}{1-r^2y},\\
  A_3(x)&=B_3(t)= \frac{1}{(1-y)^{\frac{r}{2}}}\exp\left(\left.-\sum_{n=1}^\infty \frac{y^n}{2n}\left(\frac{x^r-x^{-r}}{x-x^{-1}}\right)^{2n}\right|_{x^0}\right).
\end{align*}
\end{cor}

 Finally  we give  a conjectural formula for the last power series $G_4(w,z)$.  
 It is given in terms of  a product decomposition of
 $G_3(w,z)$.
 Note that by the results of \cite{ellingsrud2001cobordism} on $B_3$, $B_4$ and the Verlinde-Segre correspondence Corollary \ref{VerSegCor}, replacing $r$ be $-r$ replaces $G_3(w,z)$ by $\frac{1}{G_3(w,z)}$ and does not change $G_4(w,z)$. Therefore for the statement we restrict to the case  $r\ge 0$.
For $i=1,\ldots r-1$, let
$\alpha_i(y)$ be the branches of the inverse series of
\[
f(x):=\left(\frac{x^{\frac{1}{2}}-x^{-\frac{1}{2}}}{x^{\frac{r}{2}}-x^{-\frac{r}{2}}}\right)^2=x^{r-1}+\ldots.
\]
Concretely, let $g$ be the inverse series to $f(x)^{\frac{1}{r-1}}$ and put $\alpha_i(y):=g(\zeta^i y^{\frac{1}{r-1}})$ for $\zeta$ a primitive $(r-1)$-th root of unity.
Then $G_3(w,z)$ satisfies the following product formula.

\begin{prop}\label{prop:B3prod}
With $r=k-1\ge 0$ and the changes of variables
 \eqref{varchange1}, \eqref{varchange}, $y=\frac{uv}{(1-u)(1-v)}$ we have
\[
A_3(x)=B_3(t)=G_3(w,z)^2=\frac{y}{(1-y)^{r}\prod_{i=1}^{r-1} \alpha_i(y)}.
\]
\end{prop}

Now we can state the conjectural formulas for the power series  $G_4(w,z)$ in terms of the factors $\alpha_i(y)$.
  \begin{conjecture}\label{Bconj}
  With $r=k-1\ge 0$ and the changes of variables
 \eqref{varchange1}, \eqref{varchange}, $y=\frac{uv}{(1-u)(1-v)}$ we have
  \begin{align*}
  \big(G_4(w,z)G_3(w,z)^r\big)^8&=\frac{(1-r^2y)^3}{(1-y)^{3r^2}}
 \left(\prod_{i,j=1}^{r-1}(1-\alpha_i(y)\alpha_j(y))\prod_{\stackrel{i,j=1}{i\ne j}}^{r-1}(1-\alpha_i(y)^r\alpha_j(y)^r)\right)^{2},
  \end{align*}
 and $ A_4(x)=B_4(t)=G_4(w,z).$
 \end{conjecture}
As $ G_4(w,z)$ is a power series starting with $1$, this determines $G_4(w,z)$.
A higher rank generalization of the Segre and Verlinde series was studied in \cite{GottscheKool}.
Work in progress on the generating functions of these invariants, using  the  blowup formulas of \cite{GottscheBlowup} and a virtual version of
strange duality, leads to a series of conjectures about these generating functions. In the rank $1$ case, i.e. for the Hilbert schemes of points, these conjectures include the product formula for $B_3(t)$. In addition they suggest that $B_4(t)$ should  be expressible in terms
of the factors $\alpha_i(y)$ of $B_3(t)$ by means of a product formula. Using also Corollary \ref{VerSegCor}, this lead
to Conjecture \ref{Bconj}.

Below in Proposition \ref{prop:Bserieslog}, we also give an (albeit not very attractive) alternative formula for $G_4(w,z)$ in terms of binomial coefficients, which has  the advantage of being very efficiently computable. Using this alternative formula, together with computer calculations of the Verlinde numbers for $\Hilb_nS$ for $n<50$, we get the following.

 \begin{prop}\label{prop:Bconj}
 Conjecture \ref{Bconj} is true modulo $w^{50}$.
 \end{prop}

 {\bf Strategy of the proof.}
 By the product formula \eqref{prodform} (or equivalently the cobordism invariance of \cite{ellingsrud2001cobordism}) it is enough to show the result
 for $S$ a toric surface and $\alpha$ the class of a toric vector bundle. Then the localization formula expresses the generating function
 $I_{S,\alpha}(w,z)$ in terms of a "master" partition function $\Omega(w,z_1,\ldots,z_k,q,t)$. Identities of modified Macdonald polynomials lead to a functional equation for $\Omega$ which we call the symmetry. The symmetry together with some regularity
 properties leads to enough constraints on the $G_i(w,z)$, to determine them.

 It is natural to ask for the geometric meaning of the symmetry and regularity properties (see Definition \ref{defn:regularity}). The regularity seems to have the same nature as the following observation,
 which also implies  \eqref{Ispecial}  for $I_{S,\alpha}^C(x)$. Let $X$ be a compact complex manifold of dimension $d$. Let $V$ be a vector bundle of rank $k$, and $L$ a line bundle on $X$. Let
 \[
 f(z) = \chi(X, \Lambda_{-z} V\otimes L).
 \]
By Riemann-Roch one can check that
\[
\eps^{d-k} (1+\eps)^k f\left(\frac{1}{1+\eps}\right)
\]
is a polynomial in $\eps$ of degree $\leq d$. Its constant term is the Chern integral $(-1)^d \int_X c_d(V)$.

The symmetry property for instance manifests itself in the $(w,z)\leftrightarrow (w^{-1}, w z)$ symmetry of the series $G_0(w, z)G_1(w, z)$, $G_3(w, z)$, and $G_4(w, z)$, or equivalently the $u\leftrightarrow v$ symmetry after the substitution \eqref{varchange}. It's geometric meaning remains mysterious to us. One consequence is that under the assumption $\chi(\alpha)=\chi(\CO_S)=0$ we have
\[
\chi\left(\Hilb_n S, \Lambda^m \alpha^{[n]} \otimes \det(\CO_S^{[n]})^{-1}\right) = (-1)^m \chi\left(\Hilb_{m-n} S, \Lambda^m \alpha^{[m-n]} \otimes \det(\CO_S^{[m-n]})^{-1}\right).
\]


\begin{acknowledgements} L.~G.~would like to thank Don Zagier for numerous discussions and advice
over the course of several years, which were very helpful in finding some of the formulas of this paper. In the course
of these discussions he also suggested and helped to  write an efficient Pari/GP program to compute the series $B_3(w)$, $B_4(w)$ with $r$ as variable to high order. This data has many times proved to be useful
 to check ideas about the structure of these power series, and also was used to show Proposition \ref{prop:Bconj}.
 He also thanks Emanuel Carneiro for useful discussions.

 A.~M.~thanks Rahul Pandharipande for the suggestion to try to compute the Chern and Verlinde series using localization and Macdonald polynomials, and for continuous encouragement. He also thanks Dragos Oprea for his minicourse on the Segre-Verlinde correspondence and many useful discussions. A.~M.~is supported by the ERC Consolidator Grant, grant agreement 101001159.
\end{acknowledgements}

\section{Partition functions}
We will use the following ``master'' partition function
\[
\Omega(w;z_1,\ldots,z_k;q,t) = \sum_\lambda \frac{\prod_{i=1}^k\prod_{\square\in \lambda} (1-q^{c(\square)} t^{r(\square)} z_i)}{\prod_{\square\in\lambda}(q^{a(\square)+1}-t^{l(\square)})(q^{a(\square)}-t^{l(\square)+1})} w^{|\lambda|}.
\]
Here for every partition $\lambda$ we consider products over the boxes $\square\in\lambda$ and for each box we denote by $c(\square)$, $r(\square)$, $a(\square)$, $l(\square)$ the column index, the row index, the arm length and the leg length respectively. The column and row indices are zero-based.

 \subsection{Notation}
We will often use the falling and rising factorial notation
\[
a_{(n)}=a (a-1)\cdots (a-n+1),\qquad a^{(n)}=a (a+1)\cdots(a+n-1).
\]
When $f$ is a polynomial or a power series involving a variable $z$ then $f|_{z^n}$ denotes the coefficient in front of $z^n$. On the other hand, when we write $f(z)|_{z=a}$ we mean $f(a)$.

\subsection{Plethysms}
In what follows we will be using \emph{plethystic evaluation}: when $F$ is a symmetric function and $A$ is a formal expression in some variables $x,y,z,\ldots$ then to define $F[A]$ we first express $F$ in terms of the power sum functions
\[
F = f(p_1,p_2,p_3,\ldots),
\]
and then set
\[
F[A] := f(A(x,y,z,\ldots), A(x^2,y^2,z^2,\ldots), A(x^3,y^3,z^3,\ldots), \ldots).
\]
When $A$ involves letters $X$, $Y$ which stand for symmetric function alphabets we agree to treat symbols $X$, and $Y$ as sums $X=x_1+x_2+\cdots$, $Y=y_1+y_2+\cdots$. In this way we have for instance
\[
F[X] = F(x_1,x_2,\cdots),
\]
where on the left we have the plethystic evaluation, and on the right we have the usual evaluation of $F$ as a function on the arguments $x_1, x_2, \ldots$.

Denote by $\pExp$ the \emph{plethystic exponential}, which is the infinite series of symmetric functions
\[
\pExp=\exp\left(\sum_{n=1}^\infty \frac{p_n}{n}\right).
\]
It satisfies
\[
\pExp[A+B]=\pExp[A] \pExp[B],\qquad \pExp[-A]=\frac{1}{\pExp[A]}
\]
for formal expressions $A, B$ when the corresponding infinite sums make sense, and we have
\[
\pExp[x + y + z + \cdots] = \frac{1}{(1-x)(1-y)(1-z)\cdots},
\]
so one may think of $\pExp$ as a notation for product expansions.

\subsection{Macdonald polynomials toolkit}

Our main reference for Macdonald polynomials and relevant results is \cite{garsia1999explicit}. Let $\wt H_\lambda[X;q,t]$ denote the \emph{modified Macdonald polynomial} for a partition $\lambda$. This is a symmetric function in $X=(x_1,x_2,\ldots)$ whose coefficients are polynomials in $q,t$. Let
\[
B_\lambda(q,t) = \sum_{\square\in\lambda} q^{c(\square)} t^{r(\square)},\qquad D_\lambda(q,t) = -1 + (1-q)(1-t) B_\lambda,
\]
\[
T_\lambda(q,t) = q^{n(\lambda')} t^{n(\lambda)} = \prod_{\square\in\lambda} q^{c(\square)} t^{r(\square)},
\]
where $n(\lambda) = \sum_{\square\in\lambda} r(\square) = \sum_i (i-1) \lambda_i$ and $\lambda'$ denotes the conjugate partition. We also abbreviate
\[
N_\lambda(q,t) = \prod_{\square\in\lambda}(q^{a(\square)+1}-t^{l(\square)})(q^{a(\square)}-t^{l(\square)+1}).
\]

We have the \emph{Cauchy identity}
\[
\pExp\left[-\frac{XY}{(1-q)(1-t)}\right] = \sum_\lambda \frac{\wt H_\lambda[X;q,t] \wt H_\lambda[Y;q,t]}{N_\lambda(q,t)},
\]
where $X=(x_1, x_2, \ldots)$ and $Y=(y_1, y_2, \ldots)$ are two sets of variables.
\begin{rem}
	It should be clear that we can also write the left hand side as an infinite product
	\[
	\prod_{k,l=1}^\infty \prod_{i,j=0}^\infty (1-x_k y_l q^i t^j),
	\]
	but the infinite product expansion makes it not so apparent that the coefficients in front of monomials in $x$ and $y$ are rational functions in $q$ and $t$ rather then arbitrary power series.
\end{rem}

We will need the following identity proved in \cite{garsia1999explicit}:
\begin{thm}
	For any partition $\mu$ we have
	\begin{equation}\label{eq:garsia tesler}
	\wt H_\mu[X+1;q,t] = \pExp\left[\frac{X}{(1-q)(1-t)}\right] \sum_{\lambda} (-1)^{|\lambda|} \frac{\wt H_\lambda[X;q,t] \wt H_\lambda[D_\mu(q,t);q,t]}{T_\lambda(q,t) N_\lambda(q,t)}.
	\end{equation}
\end{thm}

This identity can be thought of as a strengthening of the \emph{Macdonald-Koornwinder duality}, which says that for any partitions $\mu$, $\nu$ we have the following identity of rational functions in $u, q, t$:
\begin{equation}\label{eq:macdonald koornwinder}
	\frac{H_{\nu}[1+u D_\mu(q,t); q, t]}{\prod_{\square\in\nu} (1- u q^{c(\square)}t^{r(\square)})} = 	\frac{H_{\mu}[1+u D_\nu(q,t); q, t]}{\prod_{\square\in\mu} (1- u q^{c(\square)}t^{r(\square)})}.
\end{equation}

Indeed, setting $X=u D_\nu(q,t)$ in \eqref{eq:garsia tesler}, we obtain
\[
\wt H_\mu[1+u D_\nu(q,t); q,t] = \pExp\left[\frac{-u}{(1-q)(1-t)} + u B_\nu\right] \sum_{\lambda} (-u)^{|\lambda|} \frac{\wt H_\lambda[D_\nu(q,t);q,t] \wt H_\lambda[D_\mu(q,t);q,t]}{T_\lambda(q,t) N_\lambda(q,t)},
\]
so the right hand side of \eqref{eq:macdonald koornwinder} equals
\[
\pExp\left[\frac{-u}{(1-q)(1-t)} + u B_\nu + u B_\mu\right] \sum_{\lambda} (-u)^{|\lambda|} \frac{\wt H_\lambda[D_\nu(q,t);q,t] \wt H_\lambda[D_\mu(q,t);q,t]}{T_\lambda(q,t) N_\lambda(q,t)},
\]
which is manifestly symmetric in $\mu, \nu$.

For the empty partition $\nu=\varnothing$ we have $B_\varnothing=0$, $D_\varnothing=-1$, $\wt H_\varnothing=1$, and so \eqref{eq:macdonald koornwinder} implies
\begin{equation}\label{eq:macdonald at u}
	H_\mu[1-u; q,t] = \prod_{\square\in\mu} (1- u q^{c(\square)}t^{r(\square)}).
\end{equation}

Both sides of \eqref{eq:macdonald koornwinder} are rational functions of $u$. In the limit $u\to\infty$ we obtain
\begin{equation}\label{eq:spec mk}
	(-1)^{|\nu|} \frac{H_{\nu}[D_\mu; q, t]}{T_\nu(q,t)} = 	(-1)^{|\mu|} \frac{H_{\mu}[D_\nu; q, t]}{T_\mu(q,t)}.
\end{equation}

\subsection{Functional equation}
Using these identities we can now prove the following result about the partition function $\Omega$ (see \cite{mellit2016plethystic} for a more general identity and an application to a conjecture of Hausel-Mereb-Wong \cite{hausel2019arithmetic}):
\begin{thm} We have the following identity of power series in $w,z_1,\ldots,z_k$ with coefficients in $\BQ(q,t)$:
\begin{equation}\label{eq:identity}
	\Omega(w;z_1,\ldots,z_k;q,t) = \pExp\left[-\frac{w+\sum z_i}{(1-q)(1-t)}\right]
\end{equation}
\[
\cdot \sum_\mu (-1)^{|\mu|} \frac{\tilde H_\mu[w+1;q,t] \tilde H_\mu[z_1+\cdots+z_k;q,t] T_\mu(q,t)}{N_\mu(q,t)}.
\]
\end{thm}
\begin{proof}
	Let $Z=z_1+\cdots+z_k$. We write the term in the definition of $\Omega$ as follows:
	\[
	\prod_{i=1}^k\prod_{\square\in \lambda} (1-q^{c(\square)} t^{r(\square)} z_i) = \pExp[-Z B_\lambda(q,t)] =  \pExp\left[-\frac{Z(D_\lambda(q,t)+1)}{(1-q)(1-t)}\right]
	 \]
	 \[
	  = \pExp\left[-\frac{Z}{(1-q)(1-t)}\right]\sum_{\mu} \frac{\wt H_\mu[Z;q,t] \wt H_\mu[D_\lambda(q,t);q,t]}{N_\mu(q,t)}.
	\]
	Collecting the coefficients in front of the terms $\wt H_\mu[Z;q,t]$, the statement is reduced to showing
	\[
	\sum_{\lambda} \frac{\wt H_\mu[D_\lambda(q,t);q,t]}{N_\lambda(q,t)} w^{|\lambda|} \\
	=  (-1)^{|\mu|} \pExp\left[-\frac{w}{(1-q)(1-t)}\right] \tilde H_\mu[w+1;q,t] T_\mu(q,t).
	\]
	Applying \eqref{eq:spec mk}, this reduces to
	\[
	\sum_{\lambda} \frac{\wt H_\lambda[D_\mu(q,t);q,t]}{T_\lambda(q,t) N_\lambda(q,t)} (-w)^{|\lambda|} \\
	= \pExp\left[-\frac{w}{(1-q)(1-t)}\right] \tilde H_\mu[w+1;q,t].
	\]
    Noting that $\wt H_\mu[w;q,t]=w^{|\mu|}$ we recognize the specialization of \eqref{eq:garsia tesler} to $X=w$.
\end{proof}

This implies a functional equation for $\Omega$.
\begin{cor}
	Let
	\[
	\wt \Omega(w; z_1,\ldots,z_k; q,t) = \pExp\left[\frac{w+\sum z_i}{(1-q)(1-t)}\right] \Omega(w; z_1,\ldots,z_k; q,t).
	\]
	This is a power series in $z_1,\ldots,z_k$ whose coefficients are \emph{polynomials} in $w$, and we have
	\[
	\wt \Omega(w; z_1,\ldots,z_k; q,t) = \wt \Omega(w^{-1}; w z_1,\ldots, w z_k; q,t).
	\]
\end{cor}

\subsection{The logarithm}
Let
\[
H(w;z_1,\ldots,z_k;q,t) = \log \Omega(w;z_1,\ldots,z_k;q,t).
\]
In \cite{mellit2018integrality} it was shown that series of the form
\[
\sum_\lambda \frac{\wt H_\lambda[X;q,t] \wt H_\lambda[Y;q,t] \wt H_\lambda[Z;q,t] \cdots}{N_\lambda},
\]
with arbitrary number of Macdonald polynomials in the numerator, have the property that they can be written as
\[
\pExp\left[\frac{1}{(1-q)(1-t)} (\cdots)\right],
\]
where $(\cdots)$ is a series in the variables from $X$, $Y$, $Z$, \dots whose coefficients are \emph{polynomials} in $q$, $t$. Specializing $X=w$, $Y=1-z_1$, $Z=1-z_2$ and so on and applying \eqref{eq:macdonald at u} we deduce that $\Omega$ can be written in this form. In particular, from the definition of $\pExp$ we deduce the following:
\begin{prop}
	All coefficients of
	\[
	(1-q)(1-t) H(w;z_1,\ldots,z_k;q,t),
	\]
	as a power series in $w, z_1, \ldots, z_k$ are regular in a neighborhood of $q=t=1$.
\end{prop}

Thus we can expand
\begin{equation} \label{eq:log expansion}
H(w;z_1,\ldots,z_k;e^{t_1},e^{t_2}) = \sum_{d_1,d_2\geq -1} H_{d_1,d_2}(w;z_1,\ldots,z_k) t_1^{d_1} t_2^{d_2}.
\end{equation}
for some power series $H_{d_1, d_2}(w,z_1,\ldots,z_k)$.

Next we expand
\[
\log \pExp\left[\frac{w+\sum z_i}{(1-q)(1-t)}\right] = \sum_{n=1}^\infty \frac{w^n+\sum_{i=1}^k z_i^n}{n(1-q^n)(1-t^n)}
\]
and use
\[
\frac{1}{e^{tn}-1} = \sum_{d=-1}^\infty \frac{n^d B_{d+1}}{(d+1)!} t^d,
\]
where $B_0, B_1, B_2, \ldots$ are the Bernoulli numbers $1, -1/2, 1/6,\ldots$, to deduce a functional equation for $H_{d_1,d_2}$.
\begin{thm}\label{thm:symmetry}
	For each $d_1,d_2\geq -1$ let
	\[
	\wt H_{d_1, d_2}(w; z_1,\ldots,z_k) = H_{d_1, d_2}(w; z_1,\ldots,z_k) + \frac{B_{d_1+1} B_{d_2+1}}{(d_1+1)!(d_2+1)!} \left(\Li_{1-d_1-d_2}(w) + \sum_{i=1}^k \Li_{1-d_1-d_2}(z_i)\right),
	\]
	where $\Li_d(z)=\sum_{n=1}^\infty \frac{z^n}{n^d}$ is the polylogarithm function. Then we have
	\[
	\wt H_{d_1, d_2}(w; z_1,\ldots,z_k) = \wt H_{d_1,d_2}(w^{-1}; w z_1,\ldots,w z_k).
	\]
\end{thm}
In other words, expanding $\wt H_{d_1,d_2}$ as a power series in $z_1,\ldots,z_k$ each coefficient is a \emph{palindromic} polynomial in $w$.

\section{Localization formulas}

\subsection{Integrals over the surface}
Assume $S$ is a smooth compact toric surface. Let $T=(\BC^*)^2$ which naturally acts on $S$. We denote by $\CT_S$ the tangent bundle of $S$. Let $V$ be a $T$-equivariant bundle on $S$. Denote $H^*_T(\text{point})=\BC[a_1,a_2]$. Suppose there are $M$ fixed points $p_1,\ldots,p_M$. Denote the weights of the tangent space at the $i$-th fixed point by $t_1^{(i)}(a_1,a_2),t_2^{(i)}(a_1,a_2)$ and the weights of $V$ by $v_1^{(i)}(a_1,a_2),\ldots,v_k^{(i)}(a_1,a_2)$. Each weight is a $\BZ$-linear combination of $a_1,a_2$. The fundamental class of $S$ in the localized equivariant homology is given by
\[
[S] = \sum_{i=1}^M \frac{1}{t_1^{(i)} t_2^{(i)}} [p_i].
\]
Denote by $\pi:S\to \text{point}$ the natural map. By degree reasons the following sums vanish
\begin{equation}\label{eq:vanishing1}
0=\pi_*[S] = \sum_{i=1}^M \frac{1}{t_1^{(i)} t_2^{(i)}},
\end{equation}
\begin{equation}\label{eq:vanishing2}
0=\pi_*([S]\cap c_1(\CT_S)) = \sum_{i=1}^M \frac{t_1^{(i)}+t_2^{(i)}}{t_1^{(i)} t_2^{(i)}},
\end{equation}
\begin{equation}\label{eq:vanishing3}
0=\pi_*([S]\cap c_1(V)) = \sum_{i=1}^M \frac{v_1^{(i)}+\cdots+v_k^{(i)}}{t_1^{(i)} t_2^{(i)}},
\end{equation}
and numerical invariants of the pair $S,V$ are given as follows:
\begin{equation}\label{eq:localization1}
\int_S c_2(V) = \pi_*([S]\cap c_2(V)) = \sum_{i=1}^M \frac{e_2(v_1^{(i)},\ldots,v_k^{(i)})}{t_1^{(i)} t_2^{(i)}},
\end{equation}
\begin{equation}\label{eq:localization2}
\int_S c_1(V)^2 = \pi_*([S]\cap c_1(V)^2) = \sum_{i=1}^M \frac{(v_1^{(i)}+\cdots+v_k^{(i)})^2}{t_1^{(i)} t_2^{(i)}},
\end{equation}
\begin{equation}\label{eq:localization3}
\chi(S) = \pi_*([S]\cap c_2(\CT_S)) = \sum_{i=1}^M 1 = M,
\end{equation}
\begin{equation}\label{eq:localization4}
\int_S c_1(V) c_1(\CT_S) = \pi_*([S]\cap c_1(\CT_S) c_1(V)) = \sum_{i=1}^M \frac{(t_1^{(i)}+t_2^{(i)})(v_1^{(i)}+\cdots+v_k^{(i)})}{t_1^{(i)} t_2^{(i)}},
\end{equation}
\begin{equation}\label{eq:localization5}
\int_S c_1(\CT_S)^2 = \pi_*([S]\cap c_1(\CT_S)^2) = \sum_{i=1}^M \frac{(t_1^{(i)}+t_2^{(i)})^2}{t_1^{(i)} t_2^{(i)}},
\end{equation}
where $e_2$ denotes the second elementary symmetric function. In all these formulas the right hand side does not depend on $a_1$ and $a_2$.
\begin{rem}\label{rem:numbers}
	The right hand sides of \eqref{eq:localization1}--\eqref{eq:localization5} are apriori rational functions in $a_1, a_2$, but the identities imply that these functions are constant.
\end{rem}

\subsection{Integrals over the Hilbert schemes}\label{ssec:integrals}
The fixed points of $\Hilb_{*} S = \bigcup_{n=0}^\infty \Hilb_{n} S$ correspond to $M$-tuples of partitions $\lambda^{(1)},\ldots,\lambda^{(M)}$. The weights of the tangent space are given by
\[
\bigcup_{i=1}^M \bigcup_{\square\in\lambda^{(i)}} \{(a(\square)+1) t_1^{(i)} - l(\square) t_2^{(i)},\; (l(\square)+1) t_2^{(i)} - a(\square) t_1^{(i)}\}.
\]
The bundle $V$ induces a bundle $V^{[n]}$ on each $\Hilb_n S$ whose weights at a fixed point are given by
\[
\bigcup_{i=1}^M \bigcup_{\square\in\lambda^{(i)}} \{v_j^{(i)} - t_1^{(i)} c(\square) - t_2^{(i)} r(\square)\;|\; j=1,\ldots, k\}.
\]

We can compute various integrals over $\Hilb_{*}$ as in \cite{marian2017higher}.
By definition
\[
\Omega(w;z_1,\ldots,z_k;e^{-t_1},e^{-t_2})= \sum_\lambda \frac{e^{\sum_{\square\in\lambda}a(\square)t_1+l(\square)t_2} \prod_{i=1}^k\prod_{\square\in \lambda} (1-e^{-c(\square)t_1 -r(\square)t_2} z_i)}{\prod_{\square\in\lambda}(1-e^{-(a(\square)+1)t_1+l(\square)t_2})(1-e^{a(\square)t_1-(l(\square)+1)t_2})} (-w)^{|\lambda|}.
\]
If $\lambda=(\lambda_0,\ldots,\lambda_s)$ is a partition and $\lambda'=(\lambda'_0,\ldots,\lambda'_t)$
the dual partition, then for $\square=(n,m)\in \lambda$ also $\square_1=(n,\lambda_n-m-1)\in \lambda$, and similarly
$\square_2=(\lambda'_m-n-1,m)\in \lambda$, with $a(\square)=c(\square_1)$ and $l(\square)=r(\square_2)$, therefore
\[
\sum_{\square\in\lambda}a(\square)=\sum_{\square\in\lambda}c(\square), \quad \sum_{\square\in\lambda}l(\square)=\sum_{\square\in\lambda}r(\square).
\]
Therefore we  get
\begin{align*}
&I_{S,V}(w,z)=\sum_{n=0}^\infty (-w)^n \chi(\Lambda_{-z} V^{[n]}\otimes \det(\CO_S^{[n]})^{-1})=\Bigg(\sum_{\lambda^{(1)},\cdots,\lambda^{(M)}} \prod_{i=1}^M (-w)^{|\lambda^{(i)}|}
\\&
\cdot\frac{e^{\sum_{\square\in\lambda^{(i)}}c(\square)t^{(i)}_1+r(\square)t^{(i)}_2} \prod_{j=1}^k\prod_{\square\in \lambda^{(i)}} (1-e^{-c(\square)t^{(i)}_1 -r(\square)t^{(i)}_2}e^{v^{(i)}_j} z)}{\prod_{\square\in\lambda^{(i)}}(1-e^{-(a(\square)+1)t^{(i)}_1+l(\square)t^{(i)}_2})(1-e^{a(\square)t^{(i)}_1-(l(\square)+1)t^{(i)}_2})}\Bigg)\Bigg|_{a_1=a_2=0} \\
&=\Bigg(\prod_{i=1}^M\Omega(w;ze^{v^{(i)}_1},\ldots,ze^{v^{(i)}_k};e^{-t^{(i)}_1},e^{-t^{(i)}_2})\Bigg)\Bigg|_{a_1=a_2=0}.
\end{align*}
The Chern integrals are given by the generating series
\[
I^C_{S,V}(w)=\sum_{n=0}^\infty w^n \int_{\Hilb_n S} c_{2n}(V^{[n]})
\]
\[
= \left(\sum_{\lambda^{(1)},\cdots,\lambda^{(M)}} \prod_{i=1}^M \prod_{\square\in\lambda^{(i)}} \frac{w \prod_{j=1}^k (1 + v_j^{(i)} - t_1^{(i)} c(\square) - t_2^{(i)} r(\square))}{((a(\square)+1) t_1^{(i)} - l(\square) t_2^{(i)}) ((l(\square)+1) t_2^{(i)} - a(\square) t_1^{(i)})}\right)\Bigg|_{a_1=a_2=0}.
\]
Denoting
\[
\Omega^{C}(w;v_1,\ldots,v_k;t_1,t_2) = \sum_{\lambda} \prod_{\square\in\lambda} \frac{ \prod_{j=1}^k (1 + v_j - t_1 c(\square) - t_2 r(\square))}{((a(\square)+1) t_1 - l(\square) t_2) ((l(\square)+1) t_2 - a(\square) t_1)} w^{|\lambda|},
\]
we have
\[
I^C_{S,V}(w) = \left(\prod_{i=1}^M \Omega^C(w;v_1^{(i)},\ldots,v_k^{(i)};t_1^{(i)},t_2^{(i)})\right) \Bigg|_{a_1=a_2=0}.
\]

For the Verlinde series we let
\[
\Omega^{V}(w;v_1,\ldots,v_k;t_1,t_2) = \sum_{\lambda} \prod_{\square\in\lambda} \frac{e^{\sum_{i=1}^k \left(v_i-c(\square) t_1 - r(\square) t_2\right)} }{(1-e^{-(a(\square)+1) t_1 + l(\square) t_2}) (1-e^{-(l(\square)+1) t_2 + a(\square) t_1})} w^{|\lambda|},
\]
so that
\[
I^V_{S,V}(w) = \sum_{n=0}^\infty w^n \chi(S, (\det V)_{(n)} \otimes E^k)
=
\left(\prod_{i=1}^M \Omega^V(w;v_1^{(i)},\ldots,v_k^{(i)};t_1^{(i)},t_2^{(i)})\right) \Bigg|_{a_1=a_2=0}.
\]
\begin{rem}\label{rem:polynomials}
	Similarly to Remark \ref{rem:numbers}, before setting $a_1=a_2=0$, the coefficients on the  right hand sides of the above identities are guaranteed to be power series in $a_1,a_2$. Thus setting $a_1=a_2=0$ makes sense.
\end{rem}

We want more generally to consider $I_{S,\alpha}(w,z)$, $I^C_{S,\alpha}(w)$,  $I^V_{S,\alpha}(w)$ for
elements $\alpha=V-W\in K^0(S)$, where $V$ and $W$ are $T$-equivariant bundles on $S$. Denote the weights of $V$ and $W$ at the $i$-th fixpoint by $v_1^{(i)}(a_1,a_2),\ldots,v_k^{(i)}(a_1,a_2)$ and $x_1^{(i)}(a_1,a_2),\ldots,x_m^{(i)}(a_1,a_2)$.
We consider a more general version of the partition function $\Omega$. We put
\begin{align*}
\Omega(w;z_1,\ldots,z_k;y_1,\ldots,y_m;q,t) &:= \sum_\lambda \frac{\prod_{\square\in \lambda} \frac{\prod_{i=1}^k(1-q^{c(\square)} t^{r(\square)} z_i)}{\prod_{j=1}^m(1-q^{c(\square)} t^{r(\square)} y_j)}}{\prod_{\square\in\lambda}(q^{a(\square)+1}-t^{l(\square)})(q^{a(\square)}-t^{l(\square)+1})} w^{|\lambda|},\\
\Omega^{C}(w;z_1,\ldots,z_k;y_1,\ldots,y_m;t_1,t_2) &:= \sum_{\lambda}\frac{  \prod_{\square\in\lambda} \frac{\prod_{i=1}^k (1 + z_i - t_1 c(\square) - t_2 r(\square))}{  \prod_{j=1}^m (1 + y_j - t_1 c(\square) - t_2 r(\square))}}{\prod_{\square\in\lambda}((a(\square)+1) t_1 - l(\square) t_2) ((l(\square)+1) t_2 - a(\square) t_1)} w^{|\lambda|};
\end{align*}
Then
\begin{align*}
I_{S,\alpha}(w,z)&=\Bigg(\prod_{i=1}^M\Omega(w;ze^{v^{(i)}_1},\ldots,ze^{v^{(i)}_k};ze^{x^{(i)}_1},\ldots,ze^{x^{(i)}_m};e^{-t^{(i)}_1},e^{-t^{(i)}_2})\Bigg)\Bigg|_{a_1=a_2=0},\\
I^C_{S,\alpha}(w)&= \left(\prod_{i=1}^M \Omega^C(w;v_1^{(i)},\ldots,v_k^{(i)};x^{(i)}_1,\ldots,x^{(i)}_m;t_1^{(i)},t_2^{(i)})\right) \Bigg|_{a_1=a_2=0}.
\end{align*}

To reduce to the case that $\alpha$ is a vector bundle we use a well-known trick.

\begin{lem}
	Let $f(z)=\sum_{n=0}^\infty f_n z^n$ be a power series where each $f_n$ is a power series in some other variables and $f_0$ begins with $1$. Then
	\[
	\prod_{i=1}^k f(z_i)
	\]
	is a power series whose coefficients can be written in a uniform way as polynomials in $k$ and symmetric functions of $z_1,z_2,\ldots$. Moreover, the substitution $k\to m-l$ and $p_n(z_1,z_2,\ldots)\to p_n(y_1,y_2,\ldots)- p_n(x_1,x_2,\ldots)$ produces
	\[
	\frac{\prod_{i=1}^m f(y_i)}{\prod_{j=1}^l f(x_i)}.
	\]
\end{lem}
\begin{proof}
	Suppose the logarithm of $f$ is given by
	\[
	\log f(z) = \sum_{n=0}^\infty g_n z^n,
	\]
	where $g_0$ begins with $0$. Then
	\[
	\prod_{i=1}^k f(z_i) = e^{k g_0 + \sum_{n=1}^\infty p_n(z_1,z_2,\ldots) g_n},\qquad	\frac{\prod_{i=1}^m f(y_i)}{\prod_{j=1}^l f(x_i)}= e^{(m-l) g_0 - \sum_{n=1}^\infty (p_n(y_1,y_2,\ldots)-p_n(x_1,x_2,\ldots))g_n},
	\]
	from which the statement is clear.
\end{proof}

\begin{cor}\label{cor:alpha from V}
	$\Omega(w;e^{y_1}z,\ldots,e^{y_m}z;e^{x_1}z,\ldots,e^{x_l}z;q,t) $ is obtained from
	 $\Omega(w;e^{v_1}z,\ldots,e^{v_k}z;q,t) $ and $\Omega^C(w;y_1,\ldots,y_m;x_1,\ldots,x_l;t_1,t_2)$ from $\Omega^C(w;v_1,\ldots,v_k;t_1,t_2)$ by the substitutions
	 	\[
		k \to m-l,\qquad
	p_n(v_1,v_2,\ldots)\to p_n(y_1,y_2,\ldots)-p_n(x_1,x_2,\ldots).
	\]
\end{cor}

\subsection{From the master partition function to the specialized ones}

The following specializations are easy to verify term-by-term.
\begin{prop}\label{prop:specializations}
	The Chern and Verlinde functions $\Omega^C, \Omega^V$ are given by the following term-by-term limits:
	\[
	\Omega^C(w; v_1,\ldots,v_k; t_1, t_2) = \lim_{\veps\to 0} \Omega\left(-w \varepsilon^{2-k} (1+\varepsilon)^k; \frac{e^{-\veps v_1}}{1+\veps},\ldots,\frac{e^{-\veps v_k}}{1+\veps};e^{\veps t_1}, e^{\veps t_2}\right),
	\]
	\[
	\Omega^V(w; v_1,\ldots,v_k; t_1, t_2) = \lim_{\veps\to 0} \Omega\left((-1)^k w \veps^{k+1}; \veps^{-1} e^{v_1},\ldots,\veps^{-1} e^{v_k}, \veps^{-1}; e^{-t_1}, e^{-t_2}\right).
	\]
\end{prop}
\begin{proof}
	For the Chern function substitution we have
	\[
	\Omega\left(-w \varepsilon^{2-k} (1+\varepsilon)^k; \frac{e^{-\veps v_1}}{1+\veps},\ldots,\frac{e^{-\veps v_k}}{1+\veps};e^{\veps t_1}, e^{\veps t_2}\right)
	\]
	\begin{equation}\label{eq:chern substitution}
	= \sum_\lambda \frac{\veps^{(2-k) |\lambda|} \prod_{i=1}^k\prod_{\square\in \lambda} (1+\veps-e^{\veps (c(\square) t_1 + r(\square) t_2 - v_i)})}{\prod_{\square\in\lambda}(e^{\veps(a(\square)+1) t_1}-e^{\veps l(\square) t_2})(e^{\veps(l(\square)+1) t_2} - e^{\veps a(\square) t_1})} w^{|\lambda|},
	\end{equation}
	from which it is clear that the limit $\veps\to 0$ exists and equals $\Omega^C$. The formula for the Verlinde function is straightforward.
\end{proof}

\subsection{Taking the logarithms}
By the expansion  \eqref{eq:log expansion}
\[
\log \Omega(w;z_1,\ldots,z_k;e^{t_1}, e^{t_2}) = \sum_{d_1,d_2\geq -1} H_{d_1,d_2}(w;z_1,\ldots,z_k) t_1^{d_1} t_2^{d_2}
\]
we obtain
$$\log \Omega(w;ze^{v_1}_1,\ldots,ze^{v_k};e^{t_1}, e^{t_2}) = \sum_{d_1,d_2\geq -1} H_{d_1,d_2}(w;ze^{v_1}_1,\ldots,ze^{v_k}) t_1^{d_1} t_2^{d_2}.
$$
Combining with Proposition \ref{prop:specializations} with the expansion
we also obtain
\[
\log \Omega^C(w;v_1,\ldots,v_k;t_1, t_2) = \sum_{d_1,d_2\geq -1} t_1^{d_1} t_2^{d_2} H_{d_1,d_2}^C (w; v_1,\ldots,v_k),
\]
where
\begin{equation}\label{eq:from H to Hc}
H_{d_1,d_2}^C (w; v_1,\ldots,v_k) =\lim_{\veps\to 0} \veps^{d_1+d_2} H_{d_1,d_2}\left(-w \varepsilon^{2-k} (1+\varepsilon)^k; \frac{e^{-\veps v_1}}{1+\veps},\ldots,\frac{e^{-\veps v_k}}{1+\veps}\right).
\end{equation}
An analogous statement holds for the Verlinde series with
\begin{equation}\label{eq:from H to Hv}
H_{d_1,d_2}^V (w; v_1,\ldots,v_k) = (-1)^{d_1+d_2} \lim_{\veps\to 0} H_{d_1,d_2}\left((-1)^k w \veps^{k+1}; \veps^{-1} e^{v_1},\ldots,\veps^{-1} e^{v_k}, \veps^{-1}\right).
\end{equation}

Let us expand $H_{d_1,d_2}^*$ where $*$ is either empty or one of $C,V$ as a power series in $v_i$. Writing the argument $\underline w$, we mean that the arguments are $w,z$ for $*$ empty and $w$ otherwise, and we temporarily write
$H_{d_1,d_2}(w,z;v_1,\ldots,v_k)=H_{d_1,d_2}(w;ze^{v_1}_1,\ldots,ze^{v_k})$.
\begin{equation}\label{eq:taylor expansion1}
H_{-1,-1}^*(\underline w;v_1,\ldots,v_k) = C_0^*(\underline w) + C_1^*(\underline w) \sum_{i=1}^k v_i + C_2^* (\underline w) e_2(v_1,\ldots,v_k) + C_{1,1}^*(\underline w) \left(\sum_{i=1}^k v_i\right)^2 + \cdots,
\end{equation}
\begin{equation}\label{eq:taylor expansion2}
H_{-1,0}^*(w;v_1,\ldots,v_k) = D_0^*(\underline w) + D_1^*(\underline w) \sum_{i=1}^k v_i + \cdots,
\end{equation}
\begin{equation}\label{eq:taylor expansion3}
H_{-1,1}^*(\underline w;v_1,\ldots,v_k) = E^*(\underline w) + \cdots,
\end{equation}
\begin{equation}\label{eq:taylor expansion4}
H_{0,0}^*(\underline w;v_1,\ldots,v_k) = F^*(\underline w) + \cdots,
\end{equation}
where dots mean terms of higher total degree in $v_i$.
\begin{prop}\label{integrals}
For $*$ empty or one of $C,V$ we have
\[
\log I_{S,V}^*(\underline w) = C_2^*(\underline w)\int_S c_2(V) + C_{1,1}^*(\underline w)\int_S c_1(V)^2 + F^*(\underline w) \chi(S)
\]
\[
+ D_1^*(\underline w) \int_S c_1(V) c_1(T_S) + E^*(\underline w)\left(\int_S c_1(T_S)^2 - 2 \chi(S)\right).
\]
\end{prop}
\begin{proof}
	We write the identities from Section \ref{ssec:integrals} as follows:
	\[
	\log I_{S,V}^*(\underline w)
	= \left(\sum_{d_1,d_2\geq -1} \sum_{i=1}^M \left(t_1^{(i)}\right)^{d_1} \left(t_2^{(i)}\right)^{d_2} H_{d_1,d_2}^*\left(\underline w;v_1^{(i)},\ldots,v_k^{(i)}\right)\right)\Bigg|_{a_1=a_2=0}.
	\]
	The terms of negative degree in $t_1^{(i)}, t_2^{(i)}, v_1^{(i)},\ldots,v_k^{(i)}$ after summing over $i$ produce zero by \eqref{eq:vanishing1}--\eqref{eq:vanishing3}. Terms of positive degree can be ignored because they vanish after setting $a_1=a_2=0$. Thus we are left with the terms of degree zero, which are precisely given by the series $C_2^*(\underline w)$, $C_{1,1}^*(\underline w)$, $F^*(\underline w)$, $D_1^*(\underline w)$, $E^*(\underline w)$. The sum over $i$ is then evaluated using \eqref{eq:localization1}--\eqref{eq:localization5}.
\end{proof}

Recall that we can write
\[
I_{S,V}(w,z) = G_0(w,z)^{c_2(V)} G_1(w,z)^{\chi(\det(V))} G_2(w,z)^{\frac{1}{2}\chi(\CO_S)} G_3(w,z)^{c_1(V)\cdot K_S-\frac{1}{2}K_S^2} G_4(w,z)^{K_S^2},
\]
where we omit $\int_S$ in the exponents. We use
$c_1(T_S)=-K_S,$ $\chi(S) = 12 \chi(\CO_S) - K_S^2
$
and obtain
\begin{equation}
\label{GCEF}
\begin{split}
\log G_0(w,z)&=C_2(w,z),\quad \log G_1(w,z)=2C_{1,1}(w,z),\\
\log G_2(w,z)&=24(F(w,z)-2 E(w,z))-4C_{1,1}(w,z),\\
\log G_3(w,z)&=-D_1(w,z)+C_{1,1}(w,z),\\
\log G_4(w,z)&=-F(w,z) + 3 E(w,z)+\frac{1}{2}(C_{1,1}(w,z)-D_1(w,z)).
\end{split}
\end{equation}
We get the same formulas for the $A_i(w)$, $B_i(w)$ with $C_2(w,z)$, $C_{1,1}(w,z)$, $D_1(w,z)$, $F(w,z)$, $E(w,z)$ replaced
by $C^C_2(w)$, $C^C_{1,1}(w)$, $D^C_1(w)$, $F^C(w)$, $E^C(w)$ and
$C^V_2(w)$, $C^V_{1,1}(w)$, $D^V_1(w)$, $F^V(w)$, $E^V(w)$, respectively.
Note that $\Omega^V$ depends only on the sum $\sum_{i=1}^k v_i$. Thus the term $C_2^V(w)$ vanishes.

\begin{rem}\label{univalpha}
In the product formula \eqref{prodform} for $I_{S,\alpha}(z,w)$ with $\alpha\in K^0(S)$, we see by
Corollary \ref{cor:alpha from V}, and Proposition \ref{integrals} that $G_0,\ldots,G_4$ are power series in $w,z$ whose
coefficients are universal polynomials in $k=\rk(\alpha)$. In particular they are determined for all $\alpha$ and $k$ by the cases of  vector bundles of all sufficiently high ranks.
The analogous statement holds for
$I^V_{S,\alpha}(w)$, $I^C_{S,\alpha}(z,w)$. Thus, in future we can assume that $\alpha$ is the class of  a vector bundle of rank at least $3$.
\end{rem}

\section{The constraints}

\subsection{Regularity and symmetry}
In order to determine the generating functions $G_i(w,z)$, our strategy is to obtain information about the series $H_{-1,-1}$, $H_{-1,0}$, $H_{-1,1}$ and $H_{0,0}$. Besides the functional equation from Theorem \ref{thm:symmetry} we need a constraint that stems from the observation that the right hand side of \eqref{eq:from H to Hc} before taking the limit has no pole at $\veps=0$. This is clear from \eqref{eq:chern substitution}.
\begin{defn}\label{defn:regularity}
	Let $k$ be an integer or complex parameter. Let $f(w,z)=\sum_{m,n=0}^\infty f_{m,n} w^m z^n$ be a power series.
	\begin{enumerate}
	\item $f(w,z)$ is \emph{$d$-regular} (simply \emph{regular} if $d=0$) for some integer $d\geq 0$ if for all $m$ there exists a polynomial $p_m(x)$ of degree at most $2m-d$ such that for all $n\geq 0$ we have
	\[
	f_{m,n} = (-1)^n p_m(n) \binom{km}{n}.
	\]
	$d$-regularity depends on $k$, but  $k$ will always be clear from the context.
	\item $f(w,z)$ is \emph{symmetric} if
	\[
	f(w,z)=f(w^{-1},wz)
	\]
	holds.
	\end{enumerate}
\end{defn}

The regularity condition is motivated by the following
\begin{lem}
	Suppose $k>0$ is an integer and a series $f(w,z)=\sum_{m,n=0}^\infty f_{m,n} w^m z^n$ is such that for each $m$ we have $f_{m,n}=0$ as soon as $n>km$. Form a new series
	\[
	g(w,\veps) = f\left(w \varepsilon^{2-k} (1+\varepsilon)^k, \frac{1}{1+\veps}\right).
	\]
	Then $f$ is $d$-regular if and only if $g(w,\veps)\in\veps^d\BC[[w,\veps]]$. If $f$ is $d$-regular, then we have
	\begin{equation}\label{eq:limit}
	\lim_{\veps\to 0} \veps^{-d} g(w,\veps) = (-1)^d \sum_{m=0}^\infty p_m(x)|_{x^{2m-d}} \;(km)_{(2m-d)}  w^m,
	\end{equation}
	where $p_m$ are the polynomials from Definition \ref{defn:regularity}.
\end{lem}
\begin{proof}
	Fix $m$ and consider the polynomial $f_m(z)=\sum_{n=0}^\infty f_{m,n} z^n$. Then
	\[
	g_m=g(w,\veps)|_{w^m} = f_m\left(\frac{1}{1+\veps}\right)\left(1+\veps\right)^{km} \veps^{(2-k)m}
	\]
	is a Laurent polynomial that can be written as
	\[
	g_m(\veps) = \sum_{i=0}^{km} c_i \veps^{2m-i}.
	\]
	Performing the inverse substitution $\veps=z^{-1}-1$ we find $f_m$ in terms of $g_m$:
	\[
	f_m(z)=\sum_{i=0}^{km} c_i z^i (1-z)^{km-i} = \sum_{i=0}^{km} c_i \sum_{n=0}^{km} (-1)^{n-i} \binom{km-i}{n-i} z^{n}.
	\]
	So we have
	\[
	f_{m,n} = \sum_{i=0}^{km} c_i (-1)^{n-i} \binom{km-i}{n-i} = (-1)^n \binom{km}{n}\sum_{i=0}^{km} (-1)^i \frac{(n)_{(i)}}{(km)_{(i)}} c_i
	=(-1)^n \binom{km}{n} p_m(n),
	\]
	where
	\[
	p_m(x) = \sum_{i=0}^{km} (-1)^i \frac{(x)_{(i)}}{(km)_{(i)}} c_i.
	\]
	The polynomial $p_m(x)$ is unique among polynomials of degree at most $km$ satisfying this property. The degree of $p_m$ is at most $2m-d$ if and only if $g_m\in\veps^{d} \BC[[\veps]]$, as claimed. If this is the case, then the top degree coefficient of $p_m$ is given by
	\[
	p_m(x)|_{x^{2m-d}} = (-1)^d \frac{c_{2m-d}}{(km)_{2m-d}}.
	\]
	Reversing this we obtain
	\[
	\lim_{\veps\to 0} \veps^{-d} g(w,\veps) = c_{2m-d} = (-1)^d p_m(x)|_{x^{2m-d}} (km)_{2m-d},
	\]
	as claimed.
\end{proof}

\begin{rem}
	Notice that if $f(w,z)$ is $d$-regular for a positive integer value $k$, then the assumption $f_{m,n}=0$ for $n>km$ holds automatically.
\end{rem}

\subsection{The Chern limit}
The above motivates
\begin{defn}
	Suppose a series $f(w,z)$ is $d$-regular. Its \emph{Chern limit} $f_\chern(w)$ is the series on the right hand side of \eqref{eq:limit}.
\end{defn}

In relation to our series $H_{d_1,d_2}$ we have
\begin{prop}\label{prop:properties}
	For $d_1,d_2\geq -1$ such that $d=-d_1-d_2\geq 0$ the series
	\[
	H_{d_1,d_2,k}(w,z) := H_{d_1,d_2}(w;\underbrace{z,\ldots,z}_{k})
	\]
	 is $d$-regular. The series
	\[
	H_{-1,-1,k}(w,z) + (\Li_{3}(w) + k \Li_{3}(z)),
	\]
	\[
	H_{-1,0,k}(w,z) - \frac12 (\Li_{2}(w) + k \Li_{2}(z)),
	\]
	\[
	H_{-1,1,k}(w,z) + \frac1{12} (\Li_{1}(w) + k \Li_{1}(z)),
	\]
	\[
	H_{0,0,k}(w,z) + \frac14 (\Li_{1}(w) + k \Li_{1}(z)),
	\]
	are symmetric.
\end{prop}

\begin{rem}\label{rem:recursion}
These constraints uniquely determine $H_{d_1,d_2}$ when $d=-d_1-d_2>0$. To see this, suppose we want to determine the coefficients of $w^m z^n$ for a fixed $m$. By the regularity property it is sufficient to determine these coefficients for $n=0,1,\ldots,2m-d$ since these many values completely determine a polynomial of degree at most $2m-d$. But if $n$ is one of those arguments then $n<2m$. By the symmetry property the coefficient of $w^m z^n$ is determined by the coefficient of $w^{n-m} z^n$ where $n-m<m$, so we can recursively determine all the coefficients. In the case $d=0$ the only terms whose coefficients are not determined from the recursion are $w^m z^{2m}$, and these coefficients can be arbitrary.
\end{rem}

\subsection{Regular symmetric series}\label{regsymser}
It should be clear from Remark \ref{rem:recursion} that a symmetric $d$-regular series with $d>0$ must be zero. Recall that regular means $0$-regular.
\begin{notation} We will in future write
$D_z$ for $z\frac{\partial}{\partial z} $ and $D_w$ for $w\frac{\partial}{\partial w} $.
\end{notation}

\begin{prop}\label{CDEFkprop}
	The following series are symmetric and regular:
	\[
	\CC_k(w,z):=D_z D_w\left(D_z-D_w\right) H_{-1,-1,k}(w,z),
	\]
	\[
	\CD_k(w,z):=D_z\left(H_{-1,0,k}(w,z) + \frac12 D_z H_{-1,-1,k}(w,z)\right),
	\]
	\[
	\CE_k(w,z):=H_{-1,1,k}(w,z) + \frac{1}{12} \left(D_w\left(D_z-D_w\right)-\left(D_z\right)^2\right)  H_{-1,-1,k}(w,z),
	\]
	\[
	\CF_k(w,z):=H_{0,0,k}(w,z) + \frac{1}{4} \left(D_w\left(D_z-D_w\right)-\left(D_z\right)^2\right)  H_{-1,-1,k}(w,z).
	\]
\end{prop}
\begin{proof}
	The operators $D_z$ and $D_w\left(D_z-D_w\right)$ respect symmetry and send $d$-regular series to $(d-1)$-regular series. It remains to check the polylogarithm corrections to the symmetry of the above linear combinations cancel out.
\end{proof}

\subsection{Verlinde limit}
In order to understand the specialization \eqref{eq:from H to Hv} we need
\begin{defn}
	Suppose $k>2$ is an integer and let $f(w,z)=\sum_{m,n=0}^\infty f_{m,n} w^m z^n$ be a $d$-regular series. The \emph{Verlinde limit} is defined by
	\[
	f_{\verlinde}(w) = \sum_{m=0}^\infty f_{m,km} w^m.
	\]
\end{defn}
We have
\begin{prop}\label{prop VerLim}
	The Verlinde limits of the series $\CC_{k}$, $\CD_{k}$, $\CE_{k}$, $\CF_{k}$ are given as follows:
	\begin{equation}\label{eq:C k verlinde}
\CC_{k\;\verlinde}(w) = 2k(k-1) D_w C_{1,1}^V\left((-1)^{k-1}w\right),
\end{equation}
	\begin{equation}\label{eq:D k verlinde}
\CD_{k\;\verlinde}(w) = -k D_1^V\left((-1)^{k-1}w\right) + k^2 C_{1,1}^V\left((-1)^{k-1}w\right),
\end{equation}
\[
\CE_{k\;\verlinde}(w) = E^V\left((-1)^{k-1}w\right) - \frac16 (k^2-k+1) C_{1,1}^V\left((-1)^{k-1}w\right),
\]
\[
\CF_{k\;\verlinde}(w) = F^V\left((-1)^{k-1}w\right) - \frac12 (k^2-k+1) C_{1,1}^V\left((-1)^{k-1}w\right),
\]
where in $C_{1,1}^V$, $D_1^V$, $E^V$, $F^V$ we need to use $k-1$ instead of $k$.
\end{prop}
\begin{proof}
	The Verlinde series $H_{d_1,d_2}^V(w;v_1,\ldots,v_{k})$ depends only on the sum $\sum_{i=1}^k v_k$
	\[
	H_{d_1,d_2}^V(w;v_1,\ldots,v_{k}) = H_{d_1,d_2,k}^V\left(w;\sum_{i=1}^k v_i\right)
	\]
	and from \eqref{eq:from H to Hv} setting $v_i=\frac{v}k$ we obtain
	\[
	H_{d_1,d_2,k}^V(w;v) = (-1)^{d_1+d_2} H_{d_1,d_2,k+1}\left((-1)^{k} w e^v,z\right)_\verlinde.
	\]
	So we have
	\[
	C_{1,1}^V(w) = \frac12 \left(D_w\right)^2 H_{-1,-1,k+1\;\verlinde}\left((-1)^{k} w\right),
	\]
		Using
	\[
	\left(D_w f(w,z)\right)_\verlinde = D_w \left(f(w,z)_\verlinde\right),\quad \left(D_z f(w,z)\right)_\verlinde = k D_w \left(f(w,z)_\verlinde\right)
	\]
	we obtain
	\[
	\CC_{k\;\verlinde}(w) = k(k-1) \left(D_w\right)^3 H_{-1,-1,k\;\verlinde}(w),
	\]
	Combining these identities the statement follows.
\end{proof}

\begin{rem}
	By Remark \ref{rem:recursion} a symmetric regular series $f(w,z)$ is determined by the coefficients of $w^m z^{2m}$. It is not hard to see that the three pieces of information:
	\begin{enumerate}
		\item the coefficients of $w^m z^{2m}$,
		\item the Chern limit,
		\item the Verlinde limit
	\end{enumerate}
must be related by invertible upper-triangular linear transformations independent of $f$. Thus in particular the $4$ Chern series $C_2^C(w), D_1^C(w), E^C(w), F^C(w)$ for $k$ determine the Verlinde series $C_2^V(w), D_1^V(w), E^V(w), F^V(w)$ and vice versa. In Section \ref{sec:symmetric regular functions} we determine these linear transformations precisely.
\end{rem}

\subsection{One more series}
We add the following series to the consideration:
	\[
\CC_k'(w,z):=\left(z_1\frac{\partial}{\partial z_1} z_2\frac{\partial}{\partial z_2} H_{-1,-1}(w;z_1,\ldots,z_k)\right)\Bigg|_{z_1=\cdots=z_k=z}.
\]
By a direct computation one can check that this series is symmetric and regular. By definition \eqref{eq:taylor expansion1} we have
\begin{equation}\label{eq:C prime and C}
\CC_{k}'(w,z) = C_2(w,z) + 2 C_{1,1}(w,z).
\end{equation}

\subsection{Conclusions}
The logarithms of some of the universal functions  $G_i(w,z)$ are linear combinations of regular symmetric series and thus symmetric and regular themselves.
We find from the definitions \eqref{eq:taylor expansion1}--\eqref{eq:taylor expansion4}
\[
\begin{split}
H_{0,0,k}(w,z)&=F(w,z),\quad H_{-1,1,k}(w,z)=E(w,z),\quad
 D_zH_{-1,0,k}(w,z)=-kD_1(w,z),\\
D_z^2H_{-1,-1,k}(w,z)&=2k^2C_{1,1}(w,z)+2\binom{k}{2} C_2(w,z).
\end{split}
\]
Therefore the following series are symmetric and regular
\[\begin{split}
 \CD_k(w,z)&=
-kD_1(w,z)+k^2C_{1,1}(w,z)+\binom{k}{2} C_2(w,z),\\
3\CE_k(w,z)-\CF_k(w,z)&=3E(w,z)-F(w,z).
\end{split}
\]
Thus, by \eqref{GCEF}, the following universal series are   symmetric and regular
\begin{align*}
\log (G_0(w,z)&G_1(w,z))=2C_{1,1}(w,z)+C_2(w,z)={\mathcal C}'_{k}(w,z),\\
 \log G_3(w,z)&=-D_1(w,z)+C_{1,1}(w,z)=\frac{1}{k}\CD_k(w,z)-\frac{k-1}{2}{\mathcal C}'_{k}(w,z),\\
 \log G_4(w,z)&=3\CE_k(w,z)-\CF_k(w,z)+\frac{1}{2} (C_{1,1}(w,z)-D_1(w,z)).
 \end{align*}

\section{Symmetric regular functions}\label{sec:symmetric regular functions}
We completely classify symmetric and regular series by the following:
\begin{thm}\label{thm:symmetric regular}
	Let $k$ be an integer or complex parameter. Suppose a series $f(w,z)\in\BC[[w,z]]$ is symmetric and regular. Then there exists a unique power series $h(y)\in\BC[[y]]$ such that
	\begin{equation}\label{eq:f to h}
	f\left(\frac{u(1-u)^{k-1}}{v(1-v)^{k-1}}, \frac{v}{(1-u)^{k-1}}\right) = h\left(\frac{uv}{(1-u)(1-v)}\right).
	\end{equation}
	Conversely, for any $h(y)\in\BC[[y]]$ there exists a unique power series $f$ such the above identity holds and this $f$ is symmetric and regular.
\end{thm}
\begin{proof}
	It may be worth pointing out that the main difficulty was to find the suitable two-variable substitution above. Once the substitution was discovered using computer experiments the proof is straightforward.

	A term of the form $w^m z^{m+n}$ in $f$ corresponds to the following function in $u,v$:
	\[
	u^m v^n (1-u)^{-(k-1)n} (1-v)^{-(k-1)m},
	\]
	from which it is clear that for a given $h$ there exists a unique $f$ satisfying \eqref{eq:f to h}. Moreover, this $f$ is clearly symmetric. Let us verify that it is also regular. It is sufficient to consider the case $h(y)=y^a$ for some $a\geq 0$. Then we have
	\[
	f\left(\frac{u(1-u)^{k-1}}{v(1-v)^{k-1}}, \frac{v}{(1-u)^{k-1}}\right) = \left(\frac{uv}{(1-u)(1-v)}\right)^a.
	\]
	The coefficient of $w^m z^n$ is given by the double residue
	\[
	\res_{u=v=0} \left(\frac{uv}{(1-u)(1-v)}\right)^a \left(\frac{u}{(1-v)^{k-1}}\right)^{-m} \left(\frac{v}{(1-u)^{k-1}}\right)^{m-n}
	\]
	\[
	 d\log \left(\frac{u}{(1-v)^{k-1}}\right)\wedge d\log \left(\frac{v}{(1-u)^{k-1}}\right)
	\]
	\[
	= \res_{u=v=0} (1-u)^{-a+(n-m)(k-1)} (1-v)^{-a+m(k-1)} u^{a-m} v^{a+m-n}
	\]
	\[
	\left(1-(k-1)^2\frac{uv}{(1-u)(1-v)}\right) \frac{du}{u} \wedge \frac{dv}{v}.
	\]
	So it is explicitly given as follows:
	\[
	f(w,z)|_{w^m z^n} = (-1)^n\binom{-a+(n-m)(k-1)}{m-a} \binom{-a+m(k-1)}{n-m-a}
	\]
	\[
	- (k-1)^2 (-1)^n \binom{-a-1+(n-m)(k-1)}{m-a-1} \binom{-a-1+m(k-1)}{n-m-a-1}
	\]
	\[
	=\frac{(k-2)a(n(k-1)-ak) (-1)^n}{(-a+(n-m)(k-1))(-a+m(k-1))}\binom{-a+(n-m)(k-1)}{m-a} \binom{-a+m(k-1)}{n-m-a}.
	\]
	Suppose $m\geq a$, otherwise the coefficients vanish. Notice that $\binom{-a+(n-m)(k-1)}{m-a}$ is a polynomial in $n$ of degree $m-a$. When we multiply it by $\frac{n(k-1)-ak}{(-a+(n-m)(k-1))}$ it remains a polynomial of the same degree: either $m>a$ and it is clear, or $m=a$ and the factor is $1$. We have
	\[
	\binom{-a+m(k-1)}{n-m-a} = \binom{mk}{n}\frac{(n)_{(m+a)}}{(mk)_{(m+a)}},
	\]
	where $(n)_{(m+a)}$ is a polynomial of degree $m+a$. Thus for
	\begin{equation}\label{eq: p_m for h}
	p_m(x) = \frac{(k-2)a(x(k-1)-ak)}{(-a+(x-m)(k-1))(-a+m(k-1))}\binom{-a+(x-m)(k-1)}{m-a} \frac{(x)_{(m+a)}}{(mk)_{(m+a)}}
	\end{equation}
	the condition in Definition \ref{defn:regularity} is satisfied.

	So we have an injective map that sends arbitrary power series $h$ to the corresponding symmetric regular series $f$. It remains to observe that by choosing $h$ we can achieve arbitrary values for $f(w,z)|_{w^m z^{2m}}$ and using Remark \ref{rem:recursion} conclude that this map is also surjective.
\end{proof}

\subsection{Chern and Verlinde series}
Let us compute the relationship between $h$ and the Chern/Verlinde limits of $f$.
\begin{thm}\label{thm:chern verlinde}
	Suppose $h(y)=\sum_{a=1}^\infty h_a y^a$ corresponds to $f(w,z)$ via Theorem \ref{thm:symmetric regular}. Then we have
	\[
		f_\chern\left(\frac{t}{(1+(k-1)t)^{k-1}}\right) = h\left(\frac{t}{1+(k-1)t}\right),
	\]
\[
f_\verlinde\left((-1)^k \frac{q}{(1+q)^{(k-1)^2}}\right) = h\left(\frac{q}{1+q}\right).
\
\]
Equivalently, we have
\[
f_\chern\left(y(1-(k-1)y)^{k-2}\right) = f_\verlinde\left((-1)^k y (1-y)^{k (k-2)}\right) = h(y).
\]
\end{thm}
\begin{proof}
Suppose $h(y)=y^a$. By \eqref{eq: p_m for h}, the top degree coefficient of $p_m(x)$ is given by
\[
\frac{(k-2) (k-1)^{m-a} a}{(-a+m(k-1)) (m-a)! (mk)_{(m+a)}}.
\]
To obtain the coefficient of $f_\chern(w)$ we need to multiply it by $(mk)_{(2m)}$. So we obtain
\[
f_\chern(w) = (k-2) a \sum_{m=a}^\infty \frac{(k-1)^{m-a}}{mk-m-a} \binom{mk-m-a}{m-a} w^m
\]
\[
= a \sum_{m=a}^\infty \frac{(k-1)^{m-a}}{m} \binom{mk-m-a-1}{m-a} w^m.
\]
Using residues, this can be written as follows:
\[
\res_{w=0} \left(D_w f_\chern(w)\right) w^{-m} \frac{dw}{w} = a \res_{t=0} (1+(k-1)t)^{mk-m-a-1} t^{a-m} \frac{dt}{t}.
\]
The left hand side can be written as
\[
\res_{w=0} w^{-m} d f_\chern(w).
\]
Using $y(t)=\frac{t}{1+(k-1)t}$ and $w(t)=\frac{t}{(1+(k-1)t)^{k-1}}$ the right hand side can be written as
\[
\res_{t=0} w(t)^{-m} d y(t)^a.
\]
So we have
\[
\res_{t=0} w(t)^{-m} d f_\chern(w(t)) = \res_{t=0} w(t)^{-m} d y(t)^a,
\]
and since this holds for every $m\geq 1$, we conclude $f_\chern(w(t)) = y(t)^a = h(y(t))$. Since this holds for every $a$, the statement holds for every $h$.

For the coefficient of $f_\verlinde(w)$ we simply set $n=km$ in the formula for $f(w,z)|_{w^m z^n}$ and obtain
\[
\frac{k(k-2) a (-1)^{km}}{-a+m(k-1)^2} \binom{-a+m(k-1)^2}{m-a}.
\]
So
\[
f_\verlinde(w) = k (k-2) a \sum_{m=a}^\infty \frac{(-1)^{km}}{m(k-1)^2-a} \binom{m(k-1)^2-a}{m-a} w^m
\]
\[
 = a \sum_{m=a}^\infty \frac{(-1)^{km}}{m} \binom{m(k-1)^2-a-1}{m-a} w^m.
\]
Then the proof is completed analogously to the Chern case.
\end{proof}

\subsection{The Verlinde-Segre correspondence}
The Verlinde-Segre correspondence Corollary \ref{VerSegCor}, now follows quite easily.
By the results of Section \ref{regsymser}, the series $G_0(w,z)G_1(w,z)$, $G_3(w,z)$ and $G_4(w,z)$ are symmetric and regular.
Therefore  Theorem \ref{thm:symmetric regular} implies that with the variable changes
$$
w=\frac{u(1-u)^{k-1}}{v(1-v)^{k-1}},\quad z=\frac{v}{(1-u)^{k-1}},\quad y=\frac{uv}{(1-u)(1-v)},
$$
 we can write
 $$
 G_0(w,z)G_1(w,z)=h_1(y),\quad G_3(w,z)=h_2(y),\quad G_4(w,z)=h_3(y)
 $$
  for power series $h\in \BC[[y]]$. By Theorem \ref{thm:chern verlinde}
we have with $r=k-1$ and  $x=-y(1-(k-1)y)^{k-2}$, $t=-y(1-y)^{r^2-1}$ that
 $$B_1(t)=A_0(x)A_1(x)=h_1(y).\quad B_3(t)=A_3(x)=h_2(y),\quad B_4(t)=A_4(x)=h_3(y).$$
 This shows Corollary \ref{VerSegCor}.

\section{Determination of $G_0(w,z)$, $G_1(w,z)$, $G_2(w,z)$, $G_3(w,z)$}

\subsection{An approach}\label{ssec:approach}
If $\wt f(w,z)$ is a $1$-regular series, then $f(w,z)=D_z \wt f(w,z)$ is a regular series. Let $p_m(x)$ be the polynomials corresponding to $f$ via Definition \ref{defn:regularity}. Then we have
\[
\wt f(w,0) = \sum_{m=1}^\infty w^m \lim_{x\to 0} \frac{p_m(x)}{x}.
\]
By passing to the limit in \eqref{eq: p_m for h} we obtain
\[
\wt f(w,0) = - \sum_{a,m}
\frac{a^2 (k-2)}{m} \frac{(m+a-1)_{(2a-1)}}{(m(k-1)+a)_{(2a+1)}} h_a w^m,
\]
where $h=\sum_a h_a y^a$ corresponds to $f$ via Theorem \ref{thm:symmetric regular}. We will use this identity in reverse: $\wt f(w,0)$ will be given, and we will check that our candidate for $h$ is correct.

\subsection{Explicit expressions I}
\begin{prop}\label{prop:C k explicit}
	\begin{enumerate}
	\item
	The series $\CC_k(w,z)$ via Theorem \ref{thm:symmetric regular} corresponds to
	\[
	h(y) = - \frac{k(k-1) y}{1-(k-1)^2y},
	\]
	\item
	the series $\CC_k'(w,z)$ via Theorem \ref{thm:symmetric regular} corresponds to
	\[
	h(y) = \log(1-y).
	\]
	\end{enumerate}
\end{prop}
\begin{proof}

	{\it(i)} We apply Section \ref{ssec:approach} to the situation
	\[
	\wt f(w,z) =  D_w\left(D_z-D_w\right) H_{-1,-1,k}(w,z),\qquad f(w,z) = \CC_k(w,z).
	\]
	The series $\wt f(w,z)-\Li_1(w)$ is symmetric, which implies $\wt f(w,0) = \Li_1(w)$. So it is sufficient to verify that for every $m>0$. we have
	\[
	-\sum_{a=1}^m \frac{a^2 (k-2)}{m} \frac{(m+a-1)_{(2a-1)}}{(m(k-1)+a)_{(2a+1)}} h_a = \frac{1}{m},
	\]
	where $h_a=-k (k-1)^{2a-1}$ is the $a$-th coefficient of $h$. So we need to verify
	\begin{equation}\label{eq:telescoping1}
	 k(k-2) \sum_{a=1}^m a^2 (k-1)^{2a-1} \frac{(m+a-1)_{(2a-1)}}{(m(k-1)+a)_{(2a+1)}} = 1.
	\end{equation}
	From
	\[
	(k-1)^{2a-1} \frac{(m+a-1)_{(2a-1)}}{(m(k-1)+a-1)_{(2a-1)}} - (k-1)^{2a+1} \frac{(m+a)_{(2a+1)}}{(m(k-1)+a)_{(2a+1)}}
	\]
	\[
	 = k (k-2) a^2 (k-1)^{2a-1} \frac{(m+a-1)_{(2a-1)}}{(m(k-1)+a-1)_{(2a+1)}}
	\]
	we see that the sum in \eqref{eq:telescoping1} is telescoping and the result is precisely $1$.

{\it(ii)} By Theorem \ref{thm:chern verlinde} part {\it(i)}  implies
\[
\CC_{k\;\verlinde}\left((-1)^k y(1-y)^{k(k-2)}\right) = -\frac{k(k-1) y}{1-(k-1)^2 y}.
\]
Using \eqref{eq:C k verlinde} we obtain
\begin{align*}
2 C_{1,1}^V\left(-y(1-y)^{k(k-2)}\right) &= - \int \frac{y}{1-(k-1)^2 y} d \log \left(y(1-y)^{k(k-2)}\right)= \log(1-y).
\end{align*}
Recall that $C_2^V(w)=0$. So by \eqref{eq:C prime and C}, the above gives the Verlinde specialization of the symmetric regular series $\CC_k'(w,z)$. Theorem \ref{thm:chern verlinde} implies that with the changes of variables \eqref{varchange} and $y=\frac{uv}{(1-u)(1-v)}$ we have
\[
\CC_{k}'\left(w,z\right) = \CC_{k\;\verlinde}'\left((-1)^k y(1-y)^{k(k-2)}\right) = \log(1-y).
\]
\end{proof}

\subsection{Explicit expressions II}
\begin{prop} \label{prop:D k explicit}
	Let $k\geq 1$. The series $\CD_k(w,z)$ via Theorem \ref{thm:symmetric regular} corresponds to $h(y)=\sum_{a=1}^\infty h_a y^a$ where
	\[
	h_a =-\frac{k}{2a} \left(\frac{x^{k-1}-x^{1-k}}{x-x^{-1}}\right)^{2a}\Bigg|_{x^0}.
	\]
\end{prop}
\begin{proof}
	First rewrite the constant term as follows:
	\[
	\left(\frac{1-x^{2k-2}}{1-x^2}\right)^{2a}\Bigg|_{x^{2a(k-2)}} = \left(\frac{1-x^{k-1}}{1-x}\right)^{2a}\Bigg|_{x^{a(k-2)}}
	\]
	\[
	= \sum_{i=0}^{2a} (-1)^i \binom{2a}{i} (1-x)^{-2a}\Big|_{x^{a(k-2)-i(k-1)}} = \sum_{i=0}^{a-1} (-1)^i \binom{2a}{i} \binom{ak-i(k-1)-1}{2a-1}.
	\]
	We have truncated the summation  because for $i\geq a$ we have $a(k-2)-i(k-1)\leq -a<0$. Next we replace $i$ by $a-i$ to obtain
	\[
	h_a = -k \sum_{i=1}^a (-1)^{a-i} \frac{(i(k-1)+a-1)_{(2a-1)}}{(a-i)!(a+i)!}.
	\]
	Notice that the latter expression makes sense for arbitrary $k$, so we may use it to define $h_a$ for all $k$. In order to apply Section \ref{ssec:approach}, we let
	\[
	\wt f(w,z) = H_{-1,0,k}(w,z) + \frac12 D_z H_{-1,-1,k}(w,z),
	\]
	so that $f(w,z)=D_z \CD_k(w,z)$. We have $\wt f(w,0)=\frac12 \Li_2(w)$. So it is sufficient to verify that for every $m>0$ we have
	\begin{equation}\label{eq:h a check}
	-\sum_{a=1}^m \frac{a^2 (k-2)}{m} \frac{(m+a-1)_{(2a-1)}}{(m(k-1)+a)_{(2a+1)}} h_a = \frac{1}{2 m^2}.
	\end{equation}
	We claim that for any $0<i<m$ the following holds:
	\begin{equation}\label{eq:telescoping2}
	\sum_{a=i}^m (-1)^a a^2 \frac{(i(k-1)+a-1)_{(2a-1)} (m+a-1)_{(2a-1)}}{(a-i)!(a+i)! (m(k-1)+a)_{(2a+1)}} = 0.
	\end{equation}
	Notice that
	\[
	\frac{(i(k-1)+a-1)_{(2a-1)} (m+a-1)_{(2a-1)}}{(a-i-1)!(a+i-1)! (m(k-1)+a-1)_{(2a-1)}} + \frac{(i(k-1)+a)_{(2a+1)} (m+a)_{(2a+1)}}{(a-i)!(a+i)! (m(k-1)+a)_{(2a+1)}}
	\]
	\[
	= a^2 (m^2-i^2) k (k-2) \frac{(i(k-1)+a-1)_{(2a-1)} (m+a-1)_{(2a-1)}}{(a-i)!(a+i)! (m(k-1)+a)_{(2a+1)}},
	\]
	where for $a=i$ the first summand on the left hand side is understood as $0$. Thus the sum \eqref{eq:telescoping2} telescopes to zero. This implies that when we expand \eqref{eq:h a check} into a sum over $a$ and $i$ only the term with $a=i=m$ survives. This term equals
	\[
	\frac{m^2(k-2)}{m} \frac{(2m-1)!}{(km)_{(2m+1)}} \cdot k  \frac{(km-1)_{2m-1}}{(2m)!} = \frac{1}{2m^2},
	\]
	which is the right hand side of \eqref{eq:h a check}.
\end{proof}
By Section \ref{regsymser} we have $\log G_3(w,z)=\frac{1}{k}\CD_k(w,z)-\frac{k-1}{2}\CC_k'(w,z)$.
Thus by Propositions \ref{prop:C k explicit}, \ref{prop:D k explicit} we get with the changes of variables \eqref{varchange} and $y=\frac{uv}{(1-u)(1-v)}$ that
$$G_3(w,z)=(1-y)^{-\frac{k-1}{2}}\exp\left(\sum_{n=1}^{\infty}-\frac{y^n}{2n}\left.\left(\frac{x^{k-1}-x^{1-k}}{x-x^{-1}}\right)\right|_{x^0}\right).$$

\subsection{Solving differential equations}
  We start with some preliminaries.
 Recall that we denote
 $D_w:=w\frac{\partial}{\partial w}$, $D_z=z\frac{\partial}{\partial z}$. We will also write $D_w^{-1}$ and  $D_z^{-1}$ for the inverse operations with zero integration constants.
  We freely use the changes of variables
 \begin{equation}\label{varchange2} y=\frac{uv}{(1-u)(1-v)}, \qquad w=\frac{u(1-u)^{k-1}}{v(1-v)^{k-1}}, \qquad z=\frac{v}{(1-u)^{k-1}}.\end{equation}
 We need to compute partial derivatives
 $D_w f(u,v)$, $D_z f(u,v)$ of power series  $f(u,v)$ with the result expressed in terms of $u,v$.
 For this we use the following formulas.
\begin{lem}\label{DwDz}
\begin{align*}D_w&=\frac{u(1-u)(1-v)\frac{\partial}{\partial u}-(k-1)uv(1-v)\frac{\partial}{\partial v}}{1-u-v-(k^2-2k)uv},\\
D_z&=\frac{u(1-u)(1-kv)\frac{\partial}{\partial u}+v(1-v)(1-ku)\frac{\partial}{\partial v}}{1-u-v-(k^2-2k)uv}.
\end{align*}
\end{lem}
 \begin{proof} The chain rule gives
 $$\frac{\partial f}{\partial u}=\frac{\partial f}{\partial w}\frac{\partial w}{\partial u} +\frac{\partial f}{\partial z}\frac{\partial z}{\partial u},\quad
 \frac{\partial f}{\partial v}=\frac{\partial f}{\partial w}\frac{\partial w}{\partial v} +\frac{\partial f}{\partial z}\frac{\partial z}{\partial v}.$$
 Solving for $\frac{\partial f}{\partial w}$ and $\frac{\partial f}{\partial z}$ gives
 $$\frac{\partial f}{\partial w}=\frac{(\frac{\partial z}{\partial u})^{-1}\frac{\partial f}{\partial u}-(\frac{\partial z}{\partial v})^{-1}\frac{\partial f}{\partial v}}{(\frac{\partial z}{\partial u})^{-1}\frac{\partial w}{\partial u}-(\frac{\partial z}{\partial v})^{-1}\frac{\partial w}{\partial v}},\qquad
 \frac{\partial f}{\partial z}=\frac{(\frac{\partial w}{\partial u})^{-1}\frac{\partial f}{\partial u}-(\frac{\partial w}{\partial v})^{-1}\frac{\partial f}{\partial v}}{(\frac{\partial w}{\partial u})^{-1}\frac{\partial z}{\partial u}-(\frac{\partial w}{\partial v})^{-1}\frac{\partial z}{\partial v}}.$$
 and, using \eqref{varchange2}, this gives the lemma by direct computation.
  \end{proof}

  \begin{rem}\label{symrem}
  Denote by ${\mathcal S}$ the involution ${\mathcal S}:g(w,z)\mapsto g(w^{-1},zw),$ which is the identity for symmetric power series.
From the relation
$w=\frac{u(1-u)^{k-1}}{v(1-v)^{k-1}}$, $z=\frac{v}{(1-u)^{k-1}},$
we see
that ${\mathcal S}u=v$ and ${\mathcal S}v=u$, and thus ${\mathcal S}(f(u,v))=f(v,u)$.
\end{rem}

\begin{prop}\label{Hpartial}
\begin{equation}\label{Hwz}
D_wD_zH_{-1,-1,k}(w,z)=-k\log(1-u),
\end{equation}
\begin{equation}\label{Hzz}
D_z^2 H_{-1,-1,k}(w,z)=k\big(\log((1-u)^{k-1}-v)-k\log(1-u)-\log(1-v)\big),
\end{equation}
\begin{equation}\label{Hww}
D_w^2 H_{-1,-1,k}(w,z)=\log(1-u/v)-\log(1-u).
\end{equation}
\end{prop}
 \begin{proof}
 We have $wz=u/(1-v)^{k-1}$, $ z=v/(1-u)^{k-1}$,
 and so for a power series $h(w,z)\in \BC[[w,z]]$, we find that
 \begin{equation}\label{wzzero} h(0,z)=h(0,v),\quad h(w,0)=h(u/v,0).
 \end{equation}
  To prove \eqref{Hwz}, we use
 \begin{equation}\label{Ckeq}\CC_k(w,z)=(D_z-D_w)D_wD_zH_{-1,-1,k}(w,z).\end{equation}
By  Proposition \ref{prop:C k explicit} we have with the variable change \eqref{varchange2} that
 $$\CC_k(w,z)=-\frac{k(k-1)y}{1-(k-1)^2y}=-\frac{k(k-1)uv}{1-u-v- (k^2-2k)uv}
 $$
Writing
$D_wD_zH_{-1,-1,k}(w,z)=\sum_{n,m} f_{n,m}w^nz^m,$
we get by symmetry of $\CC_k(w,z)$ that
\begin{align*}\CC_k(w,z)&=\sum_{n,m}(m-n)f_{n,m}w^nz^m=\sum_{n,m}(m-n)f_{n,m}w^{m-n}z^m=\sum_{n,m}n f_{m-n,m}w^n z^m.
\end{align*}
Using Remark \ref{symrem} we have
\begin{align*}
D_{w}^{-1}\CC_k(w,z)&=\sum_{\substack{n,m\\ n\ne 0}}f_{m-n,m}w^n z^m=\sum_{\substack{n,m\\n\ne m}}f_{n,m}w^{m-n} z^m={\mathcal S}\Big(\sum_{\substack{n,m\\n\ne m}} f_{n,m}w^{n} z^m \Big).
\end{align*}
As ${\mathcal S}$ is an involution, we get
\begin{align*}{\mathcal S}\Big(D_{w}^{-1}\CC_k(w,z)+\sum_nf_{n,n}z^n\Big)&={\mathcal S}D_{w}^{-1}\CC_k(w,z)+\sum_nf_{n,n}w^nz^n\\
&=\sum_{n,m} f_{n,m}w^{n} z^m=D_wD_zH_{-1,-1,k}(w,z).\end{align*}
We write
$$H_{-1,-1,k}(w,z)+\Li_3(w)+k \Li_3(z)=\sum_{n,m} g_{n,m} w^nz^m.$$
As this is symmetric, and we have by definition $H_{-1,-1,k}(0,z)=0$,  we get
$$\sum_{n}g_{n,n}w^nz^n=\sum_{n}g_{0,n} (wz)^n=k \Li_3(zw).$$
This gives
$$\sum_n f_{n,n}z^{n}=k D_z^2 \Li_3(z)=k \Li_1(z)=-k\log(1-z).$$
 Thus we obtain
$$D_wD_zH_{-1,-1,k}(w,z)={\mathcal S}\Big(D_{w}^{-1}\CC_k(w,z)-k\log(1-z)\Big).$$
Thus the proof of \eqref{Hwz}
is reduced to the explicit identity
$$D_{w}^{-1}\Big(-\frac{(k-1)uv}{1-u-v- (k^2-2k)uv}\Big)=-{\mathcal S}(\log(1-u))+\log(1-z)=-\log(1-v)+\log(1-z).$$
By \eqref{wzzero} we have $\log(1-z)\big|_{w=0}=\log(1-v)|_{w=0}$.
Thus it is enough to see that
$$D_w (\log(1-v))=-\frac{(k-1)uv}{1-u-v- (k^2-2k)uv}.$$
Using Lemma \ref{DwDz}, this follows by a direct computation.

Now we deal with \eqref{Hzz}.
By \eqref{Ckeq} and \eqref{Hwz},  we have
\begin{align*}D_z^2 H_{-1,-1,k}(w,z)&=D_w^{-1}\CC_k(w,z)+D_wD_zH_{-1,-1,k}(w,z)\\&=D_{w}^{-1}\Big(-\frac{k(k-1)uv}{1-u-v- (k^2-2k)uv}\Big)+k\log(1-u)
\end{align*}
Thus it is enough to show the explicit identity
\begin{align*}D_{w}^{-1}\Big(-\frac{(k-1)uv}{1-u-v- (k^2-2k)uv}\Big)=\log\big((1-u)^{k-1}-v\big)-(k-1)\log(1-u)-\log(1-v).
\end{align*}
By \eqref{wzzero} we have the coefficient of $w^0$ of the right hand side is
$ \log(1-v)-\log(1-v)=0.$
Thus to show  \eqref{Hzz}, we only have to see
\begin{align*}-\frac{(k-1)uv}{1-u-v- (k^2-2k)uv}=D_w\big(\log\big((1-u)^{k-1}-v\big)-(k-1)\log(1-u)-\log(1-v)\big),
\end{align*}
which again,  using Lemma  \ref{DwDz}, follows by a direct computation.

Finally we show \eqref{Hww}.
We see that $$D_w^2 H_{-1,-1,k}(w,z)=D_z^{-1}D_w (D_zD_w H_{-1,-1,k}(w,z))-\Li_1(w),$$
thus we are reduced to the proof of the identity
$$D_z^{-1}D_w(-k\log(1-u))=\log(1-u/v)-\log(1-u)-\log(1-w).$$
By \eqref{wzzero}, the left hand side vanishes at $z=0$, and it is enough to see that
$$D_w(-k\log(1-u))=D_z\big(\log(1-u/v)-\log(1-u)\big),$$
which is again a direct computation using Lemma  \ref{DwDz}.
\end{proof}

\subsection{Determining $G_0$, $G_1$, and $G_2$}
By Proposition \ref{prop:C k explicit}  we have
\[
\log(G_0\left(w, z\right) G_1\left(w,z\right))= \CC_{k}'\left(w, z\right) = \log(1-y)=\log\left(\frac{1-u-v}{(1-u)(1-v)}\right).
 \]
By \eqref{Hzz} we have
\begin{align*}G_0(w,z)&=\exp\big(k\CC_k'(w,z)-\frac{1}{k}D_z^2H_{-1,-1,k}(w,z)\big)
=\frac{(1-u-v)^{k}}{(1-v)^{k-1}\big((1-u)^{k-1}-v\big)},
\end{align*}
which also gives
$ G_1(w, z)=\frac{(1-v)^{k-2}((1-u)^{k-1}-v)}{(1-u)(1-u-v)^{k-1}}.$

 Finally we  determine $G_2(w,z)$. Putting $K_k(w,z):=\left(D_w(D_z-D_w)-D_z^2\right)H_{-1,-1,k}(w,z),$ we get
by Proposition \ref{Hpartial} that
\begin{equation}
\label{eq1} \exp\left(K_k(w,z)\right)=\frac{(1-v)^k(1-u)^{k^2-k+1}}{(1-\frac{u}{v})\big((1-u)^{k-1}-v\big)^k}.\end{equation}
 By \eqref{GCEF} and the definition of $\CE_k(w,z)$ and $\CF_k(w,z)$ in Proposition \ref{CDEFkprop}, we have
\begin{equation}\label{eq2}
\begin{split}
\log(G_2(w,z))&=24\big(H_{0,0,k}(w,z)-2H_{-1,1,k}(w,z)\big)-4C_{1,1}(w,z)\\
&=24\big(\CF_k(w,z)-2\CE_k(w,z)\big)-2K_k(w,z)-2\log G_1(w,z).
\end{split}
\end{equation}
From \cite{ellingsrud2001cobordism} we have
\[
B_2\left(-y(1-y)^{k(k-2)}\right) = \frac{(1-y)^{(k-1)^2}}{1-(k-1)^2 y}.
\]
By the remarks after \eqref{GCEF}, Proposition \ref{prop:C k explicit}, Theorem \ref{thm:chern verlinde} and Proposition \ref{prop VerLim}, this gives
\[
24(F^V - 2 E^V)\left(-y(1-y)^{k(k-2)}\right) = \left((k-1)^2+2\right)\log(1-y) - \log(1-(k-1)^2 y),
\]
\[
24( \CF_{k\;\verlinde} - 2 \CE_{k\;\verlinde})\left((-1)^k y(1-y)^{k(k-2)}\right) = (1-k^2) \log(1-y) -  \log(1-(k-1)^2 y)
\]
As $( \CF_{k\;\verlinde} - 2 \CE_{k\;\verlinde})$ is symmetric and regular, it follows that
\begin{equation}
\label{eq3}
24\big(\CF_k(w,z)-2\CE_k(w,z)\big)=(1-k^2)\log\left(\frac{1-u-v}{(1-u)(1-v)}\right)-\log\left(\frac{1-u-v-(k^2-2k)uv}{(1-u)(1-v)}\right)
\end{equation}
Thus, combining \eqref{eq1},\eqref{eq2}, \eqref{eq3} and the formula for $G_1(w,z)$, we get
\[
G_2(w,z)=\frac{(1-\frac{u}{v})^2(1-v)^{(k-2)^2}\big((1-u)^{k-1}-v\big)^{2(k-1)} }{(1-u-v)^{(k-1)^2}(1-u)^{k^2-2k}\big(1-u-v-(k^2-2k)uv\big)}.
\]

\section{The missing power series}

We  give evidence for the conjectural formula of Conjecture \ref{Bconj} for $G_4(w,z)$. First we show the product formula  Proposition \ref{prop:B3prod}  for $G_3(w,z)$.
We begin with  a version of the Lagrange inversion formula that we will use.
\begin{prop}\label{prop:inverse1}
	Let $f(y)=y^{-1} + \cdots$ be a Laurent series. Let $g$ be the inverse series of $\frac{1}{f(y)}$. Then we have
	\[
	\frac{g(u)}{u} = \exp\left(\sum_{n=1}^\infty \frac{u^n}{n} (f(y)^n)_{y^0}\right).
	\]
\end{prop}
\begin{proof}
	Take the logarithmic derivative of both sides. For $n\geq 1$ the coefficient of $u^{n-1}$ on the left is
	\[
	\res \frac{g'(u)}{g(u)}u^{-n} du.
	\]
	Perform the substitution $u=\frac{1}{f(y)}$. Then $g(u)=y$ and the expression above equals
	\[
	\res f(y)^{n} \frac{dy}{y} = (f(y)^n)_{y^0},
	\]
	which is the term on the right.
\end{proof}

\begin{prop}\label{prop:inverse2}
	Let $f(y)=y^{-m}+\cdots$ be a Laurent series for $m\geq 1$. Let $g_1,\ldots,g_m$ be the different branches of the inverse series of $\frac{1}{f(y)}$. Then we have
	\[
	\frac{\prod_{i=1}^m g_i(u)}{u} = \exp\left(\sum_{n=1}^\infty \frac{u^n}{n} (f(y)^n)_{y^0}\right).
	\]
\end{prop}
\begin{proof}
	Let $f_1(y)=f(y)^{1/m}$. Let $g$ be the inverse series of $\frac{1}{f_1(y)}$. Then the branches $g_i$ are given by $g_i(u) = g(\zeta^i u^{1/m})$, where $\zeta$ is a primitive $m$-th root of unity. By Proposition \ref{prop:inverse1} we have
	\[
	\frac{\prod_{i=1}^m g_i(u)}{u} = \exp\left(\sum_{i=1}^m \sum_{n=1}^\infty \frac{\zeta^{in} u^{n/m}}{n} (f_1(y)^n)_{y^0}\right).
	\]
	From
	\[
	\sum_{i=1}^m \zeta^{in} = \begin{cases}
		m & (n=m n_1,\, n_1\in\BZ),\\
		0 & \text{otherwise}
	\end{cases}
	\]
	we obtain the statement.
\end{proof}

\begin{proof}[Proof of Proposition  \ref{prop:B3prod}]
Apply Proposition \ref{prop:inverse2} to $f(x):=\left(\frac{x^{\frac{1}{2}}-x^{-\frac{1}{2}}}{x^{\frac{r}{2}}-x^{-\frac{r}{2}}}\right)^2$.
Note that $f(x)^n\big|_{x^0}= f(x^2)^n\big|_{x^0}$. Thus we get
\[
\frac{y}{\prod_{i=1}^{r-1} \alpha_i(y)}=\exp\left(-\sum_{n=1}^\infty \frac{y^n}{n} \left.\left(\frac{x-x^{-1}}{x^{r}-x^{-r}}\right)^{2n}\right|_{x^0}\right)=(1-y)^r B_3(-y(1-y)^{r^2-1})^2
\]
For $G_3(w,z)$ apply Corollary \ref{VerSegCor}.
\end{proof}

We can rewrite the conjectural formula for $B_4(t)$ directly in terms of binomial coefficients.
The argument involves again Proposition \ref{prop:inverse1}. We leave the details to the reader as an exercise.
We introduce the following expressions, where the sums are over nonnegative integers.
\begingroup\allowdisplaybreaks
\begin{align*}
\alpha_n&=
\sum_{i=0}^{\lfloor\frac{n}{2}\rfloor}\sum_{j+k=n+2i}(-1)^j\binom{kr+n-1}{2n-1}\binom{2n}{j},\\
\beta_n&=\sum_{\stackrel{k+l=n}{k,l>0}}\frac{1}{kl}\sum_{i=0}^{k-1}\left(
\sum_{j=0}^{l-1}(-1)^{i+j}\binom{2k}{i}\binom{2l}{j}
\sum_{e=1}^{2l}e\cdot\binom{(r+1)l-jr}{2l-e}\binom{(r+1)k-ir}{2k+e}\right.\\&+\left.\sum_{j=l}^{\min(2l,n-i-1)}
   (-1)^{i+j}\binom{2k}{i}\binom{2l}{j}\frac{(jk - li)r}{n}\binom{(r+1)n-r(i+j)-1}{2n-1}\right),\\
 \gamma_n&=\sum_{\stackrel{k+l=n}{k,l>0}} \frac{1}{kl}\sum_{a=0}^{min(k,l)}a\Bigg(\sum_{i+j=k-a}(-1)^i\binom{rj+k-1}{2k-1}\binom{2k}{i}\Bigg)
 \\&\cdot\Bigg(\sum_{i'+j'=a+l}(-1)^{i'}\binom{rj'+l-1}{2l-1}\binom{2l}{i'}\Bigg).
\end{align*}
\endgroup



\begin{prop} \label{prop:Bserieslog}
With the above notation, Conjecture \ref{Bconj} is equivalent to the following formula for $B_4$.\begin{align*}
B_4(-y(1-y)^{r^2-1})&=
\exp\left(\sum_{n=1}^\infty \frac{y^n}{8n}\Big(4r\alpha_n-r^2-3r^{2n}-2n\beta_n-2nr^2\gamma_n\Big)\right).
\end{align*}
\end{prop}
Note that by Corollary \ref{VerSegCor} with the changes of variables \eqref{varchange} and $u=\frac{-(1-v^{-1}}{1-(1-v^{-1})y}$, we have $G_4(w,z)=B_4(-y(1-y)^{r^2-1})$.

The formula of Proposition \ref{prop:Bserieslog} for $B_4$ is not very attractive, but it has the advantage that it can be
easily evaluated by computer to high order in $w$ with $r$ as variable.  In a previous attempt to understand the power series $B_3$, $B_4$, the first named author computed with Don Zagier the power series $B_4(t)$ modulo $t^{50}$ via a Pari/GP program. This program essentially computes the lowest order terms in $a_1$, $a_2$ of the specialization
$\Omega^V(w;0,\ldots,0;t_1,t_2)$ of the power series of Proposition \ref{prop:specializations} modulo $w^{50}$ with $r=k-1$ as variable.  Using Proposition \ref{prop:Bserieslog}, a direct Pari/GP computation then  shows Proposition \ref{prop:Bconj}.

\bibliographystyle{amsalpha}
\bibliography{refs}

\end{document}